\documentclass[11pt]{article}

\usepackage{amsmath, amssymb, amsthm, verbatim,enumerate,bbm,color, mathtools,} 
\numberwithin{equation}{section}

\usepackage{float}
\usepackage{mathabx}

\usepackage{mathrsfs}

\usepackage{tikz}
\usetikzlibrary{calc}
\usetikzlibrary{math}
\usetikzlibrary{decorations,decorations.markings,decorations.pathreplacing,decorations.pathmorphing}
\usetikzlibrary{fadings}
\usetikzlibrary{fit}
\usetikzlibrary{matrix}
\usetikzlibrary{patterns}
\usetikzlibrary{positioning}
\usetikzlibrary{shadows}
\usetikzlibrary{shapes,shapes.geometric}
\usetikzlibrary{trees}
\usetikzlibrary{overlay-beamer-styles}
\usepackage{tikzsymbols}

\usepackage[colorlinks,citecolor=blue,bookmarks=true]{hyperref}
\usepackage[capitalize]{cleveref}
\usepackage[numbers,sort]{natbib}

\date{\today}
\allowdisplaybreaks

\theoremstyle{plain}
\newtheorem{theorem}{Theorem}[section]
\newtheorem{lemma}[theorem]{Lemma}
\newtheorem{claim}[theorem]{Claim}
\newtheorem{proposition}[theorem]{Proposition}

\newtheorem{conjecture}[theorem]{Conjecture}
\newtheorem{problem}[theorem]{Problem}

\theoremstyle{definition}
\newtheorem{definition}[theorem]{Definition}

\newtheorem{remark}[theorem]{Remark}

\DeclareMathOperator{\Bin}{Bin}

\newcommand\ab[1]{\lvert #1 \rvert}

\makeatletter
\def\moverlay{\mathpalette\mov@rlay}
\def\mov@rlay#1#2{\leavevmode\vtop{%
		\baselineskip\z@skip \lineskiplimit-\maxdimen
		\ialign{\hfil$\m@th#1##$\hfil\cr#2\crcr}}}
\newcommand{\charfusion}[3][\mathord]{
	#1{\ifx#1\mathop\vphantom{#2}\fi
		\mathpalette\mov@rlay{#2\cr#3}
	}
	\ifx#1\mathop\expandafter\displaylimits\fi}
\makeatother

\makeatletter
\renewenvironment{proof}[1][\proofname]
{\par\pushQED{\qed}
	\normalfont\topsep6\p@\@plus6\p@\relax\trivlist
	\item[\hskip\labelsep\bfseries#1\@addpunct{.}]
	\ignorespaces}
{\popQED\endtrivlist\@endpefalse}

\makeatother

\renewcommand{\P}{\mathcal P}

\newcommand{\poly}{\text{poly}}

\DeclareMathOperator{\twr}{twr}

\title{Is it easy to regularize a hypergraph with easy links?}
\author{Lior Gishboliner\thanks{Department of Mathematics, University of Toronto, ON, Canada. {Email}: \texttt{lior.gishboliner@utoronto.ca}. Research supported by the NSERC Discovery Grant ``Problems in Extremal and Probabilistic Combinatorics".} \and Asaf Shapira\thanks{School of Mathematics, Tel Aviv University, Israel. Email: \texttt{asafico@tau.ac.il}. Supported in part by ERC Consolidator Grant 63438 and NSF-BSF Grant 20196.} \and Yuval Wigderson\thanks{Institute for Theoretical Studies, ETH Z\"urich, Switzerland. Email: \texttt{yuval.wigderson@eth-its.ethz.ch}. Research supported by Dr.\ Max R\"ossler, the Walter Haefner Foundation, and the ETH Z\"urich Foundation.}}
\date{}

\begin{document}
\maketitle
\begin{abstract}
A partition of a (hyper)graph is $\varepsilon$-homogeneous if the edge densities between almost all clusters are either at most $\varepsilon$ or at least $1-\varepsilon$. Suppose a $3$-graph has the property that the link of every vertex has an $\varepsilon$-homogeneous partition of size
$\poly(1/\varepsilon)$. Does this guarantee that the $3$-graph also has a small homogeneous partition?
Terry and Wolf proved that such a $3$-graph has an $\varepsilon$-homogeneous partition of size given
by a wowzer-type function. Terry recently improved this to a double exponential bound, and conjectured that this bound is tight.
Our first result in this paper disproves this conjecture by giving an improved (single) exponential bound, which is best possible.
We further obtain an analogous result for $k$-graphs of all uniformities $k \geq 3$.

The above problem is part of a much broader programme which seeks to understand the conditions under which a (hyper)graph has
small $\varepsilon$-regular partitions. While this problem is fairly well understood for graphs, the situation is (as always) much more involved
already for $3$-graphs. For example, it is natural to ask if one can strengthen our first result by only requiring
each link to have $\varepsilon$-regular partitions of size $\poly(1/\varepsilon)$. Our second result shows that surprisingly the answer is `no', namely, a $3$-graph might only have regular partitions
of tower-type size, even though the link of every vertex has an $\varepsilon$-regular partition of polynomial size.
\end{abstract}

\section{Introduction}
Many important questions in extremal combinatorics are of a ``local-to-global'' nature: given local information about a certain discrete object, can we deduce something about its global structure? For example, extremal statements such as Tur\'an's theorem \cite{MR18405} or the Brown--Erd\H os--S\'os conjecture \cite{MR351888} assert that if a (hyper)graph has no local portion that is too dense, then one can deduce a stronger bound on its global density.

When dealing with hypergraphs, a natural type of local condition that one can impose concerns the \emph{links} of vertices. Here, given a $k$-uniform hypergraph $H$ and a vertex $v \in V(H)$, the link $L(v)$ of $v$ is the $(k-1)$-uniform hypergraph on vertex set $V(H) \setminus \{v\}$ and edge set $\{S: S \cup \{v\} \in E(H)\}$. The collection of all links $\{L(v):v \in V(H)\}$ provide a local view of the hypergraph $H$, and it is natural to ask which global properties can be deduced from this set of local views. 

There are a number of well-known results and problems along these lines. For example, a famous result of Garland \cite{MR320180}, which is critical to the study of high-dimensional expanders (see e.g.\ the survey \cite{MR3966743}), roughly states that if all links are good spectral expanders, then the entire hypergraph is a good expander; see \cite[Theorem 2.5]{MR3966743} for the precise statement. In Tur\'an theory, a statement of this type is given by a famous conjecture attributed to Erd\H os and S\'os (see \cite[Conjecture 1]{MR753720}), which states that if all links of a $3$-uniform hypergraph are bipartite, then the hypergraph has edge density at most $\frac 14$. In Ramsey theory, a recent result of Fox and He \cite{MR4332486} essentially completely resolves the question of what global information on the independence number of a $3$-uniform hypergraph (henceforth $3$-graph) can be deduced from the local information of forbidden configurations within its links. 

In this paper, we study analogous questions that arise in the theory of graph and hypergraph regularity. A bipartite graph with parts $A,B$ is said to be \emph{$\varepsilon$-regular} if for all $X\subseteq A, Y \subseteq B$ with $\ab{X} \geq \varepsilon \ab A, \ab{Y} \geq \varepsilon \ab B$ it holds that
$\ab{d(X,Y)-d(A,B)} \leq \varepsilon$, where $d(X,Y)\coloneqq e(X,Y)/(\ab{X}\ab{Y})$ denotes the edge density. The famous regularity lemma of Szemer\'edi \cite{MR540024} asserts that every graph has a vertex partition into parts $V_1,\dots,V_k$ such that all but an $\varepsilon$-fraction of pairs of parts $(V_i,V_j)$ define an $\varepsilon$-regular bipartite graph, and such that $k \leq K(\varepsilon)$ for some constant $K(\varepsilon)$ depending only on $\varepsilon$; such a partition is called an $\varepsilon$-regular partition of size $k$. While Szemer\'edi's regularity lemma is of extraordinary importance in graph theory, theoretical computer science, number theory, and other areas of mathematics, its applicability is often limited by the terrible quantitative bounds it produces. Indeed, Szemer\'edi's proof showed that $K(\varepsilon)$ is at most a tower-type function\footnote{We recall that the \emph{tower function} is defined by $\twr(0)=1$ and $\twr(x+1)=2^{\twr(x)}$ for $x \geq 0$.} of $1/\varepsilon$, and a famous construction of Gowers \cite{Gowers} demonstrates that such tower-type bounds are necessary in the worst case. As such, there is a great deal of interest in proving more reasonable quantitative bounds on $K(\varepsilon)$ if one assumes that $G$ has certain extra structure. Such assumptions can be global in nature (such as assuming that $G$ is defined by semi-algebraic relations of bounded complexity \cite{MR3451124,MR3585030,MR3852184}) or local in nature (such as assuming that $G$ forbids a fixed induced bipartite pattern \cite{MR2341924,FPS,MR2815610,MR4350155}). In both of these cases, one can obtain an $\varepsilon$-regular partition with only $\poly(1/\varepsilon)$ parts, a substantial improvement over the tower-type bound that is necessary without such assumptions. In this paper we focus on questions of this type in the setting of (hyper)graphs, but it is worth noting that similar questions have recently been studied in other areas such as arithmetic regularity in  groups \cite{CPT,TW20,TW21,MR3386249}.

In hypergraphs, the most natural extension of the notion of $\varepsilon$-regularity is now termed \emph{weak $\varepsilon$-regularity}. We say that a tripartite $3$-graph with parts $A,B,C$ is weakly $\varepsilon$-regular if for all $X \subseteq A, Y \subseteq B, Z \subseteq C$ with $\ab X \geq \varepsilon \ab A, \ab Y \geq \varepsilon \ab B, \ab Z \geq \varepsilon \ab C$, we have
\[
	\ab{d(X,Y,Z)-d(A,B,C)} \leq \varepsilon,
\]
where $d(X,Y,Z) \coloneqq e(X,Y,Z)/(\ab X \ab Y \ab Z)$ denotes the edge density.

Extending Szemer\'edi's work, Chung \cite{MR1099803} proved a regularity lemma for $3$-graphs, which states that every $3$-graph has a vertex partition into at most $K(\varepsilon)$ parts, such that all but an $\varepsilon$-fraction of the triples of parts define a weakly $\varepsilon$-regular tripartite $3$-graph. Her proof again yields tower-type bounds on $K(\varepsilon)$, and it is easy to embed Gowers's example into a $3$-graph to deduce that such tower-type bounds are necessary in the worst case. We remark that Chung's notion is now called \emph{weak} regularity because it is too weak to be of use in most applications of the regularity method, such as the proof of the removal lemma; for such applications, more refined notions of hypergraph regularity had to be developed, and we refer to \cite{MR2195580,MR2373376,MR2167756} for more information.

Following the pattern of local-to-global questions discussed above, it is natural to ask whether weak regularity of a $3$-graph can be deduced from regularity of its links, and it is not hard to see that the answer is yes. Indeed, let $H$ be a tripartite $3$-graph with parts $A,B,C$, and suppose that every vertex $c \in C$ has a link\footnote{Strictly speaking, the link of a vertex $c \in C$ consists of a bipartite graph between $A$ and $B$, plus isolated vertices corresponding to every vertex in $C\setminus \{c\}$. But when working in the partite setting, it is more natural to delete these isolated vertices.} $L(c)$ which is an $\varepsilon$-regular bipartite graph between $A$ and $B$. Note that for all $X \subseteq A, Y \subseteq B, Z \subseteq C$, we have
\[
	e_H(X,Y,Z) = \sum_{c \in Z} e_{L(c)}(X,Y),
\]
simply by the definition of the link. As a consequence, the edge densities satisfy
\[
	d_H(X,Y,Z) = \frac{1}{\ab Z}\sum_{c \in Z} d_{L(c)}(X,Y).
\]
Now suppose that $\ab X \geq \varepsilon \ab A,\ab Y \geq \varepsilon \ab B, \ab Z \geq \varepsilon \ab C$. Then the assumption that $L(c)$ is $\varepsilon$-regular implies that $\ab{d_{L(c)}(X,Y) -d_{L(c)}(A,B)}\leq \varepsilon$. Note too that  $\frac{1}{\ab Z} \sum_{c \in Z} d_{L(c)}(A,B)=d_H(A,B,Z)$, hence we conclude that
\[
	\ab{d_H(X,Y,Z)-d_H(A,B,Z)} \leq \varepsilon.
\]
If we now assume that every vertex $a \in A$ has an $\varepsilon$-regular link as well, we may run the same argument and deduce that $\ab{d_H(A,B,Z)-d_H(A,B,C)} \leq \varepsilon$. 

In other words, we have just proved that if all vertices\footnote{In fact, we did not even need to use this assumption for vertices of $B$. On the other hand, only assuming that vertices in one part (say $C$) have regular links does not suffice; indeed, define a 3-graph by having half of the vertices of $C$ be isolated and the other half form an edge with every pair in $A \times B$. This 3-graph is not $\varepsilon$-regular but every vertex in $C$ has an $\varepsilon$-regular link.} of $H$ have an $\varepsilon$-regular link, then $H$ itself is weakly $2\varepsilon$-regular. Of course, the utility of the regularity lemma is that it allows us to decompose \emph{any} graph into $\varepsilon$-regular pieces, so the natural next question is as follows: if every vertex-link in $H$ has a small regular partition, does this guarantee that $H$ itself has a small weakly $\varepsilon$-regular partition?

Our first result answers this question very strongly in the negative, even if we require that the links have small regular partitions with a regularity parameter much smaller than $\varepsilon$.
\begin{theorem}\label{thm:lower bound}
	Let $0 < \delta \leq \varepsilon$ such that $\varepsilon$ is sufficiently small. Then for every $n \geq n_0(\delta)$, there is an $n$-vertex tripartite $3$-graph $H$ with the following properties. 
	\begin{itemize}
		\item For all $v \in V(H)$, the link $L_H(v)$ has a $\delta$-regular partition with $O(\delta^{-8})$ parts.
		\item Every weakly $\varepsilon$-regular equipartition of $H$ has at least $\twr\left( \Omega(\log \frac 1 \varepsilon) \right)$ parts.
	\end{itemize}
\end{theorem}

In particular, as \emph{every} $3$-graph has a weakly $\varepsilon$-regular equipartition of tower-type size, \cref{thm:lower bound} states that having small regular partitions of all links is essentially useless: it provides no extra benefit, in terms of the size of a weakly regular partition, than assuming no such information. We remark that by adapting ideas from \cite{MS16} one can improve the lower bound to be of the form $\twr(\poly(\frac 1 \varepsilon))$; however, as we believe that the important issue is whether there is a sub-tower bound, we decided to stick with a simpler proof that gives only a logarithmic lower bound on the tower \nolinebreak height.

However, our other main result is positive, and says that information about regular partitions of the links \emph{is} useful, if we strengthen the notion of regularity. Namely, let us say that a bipartite graph with parts $A,B$ is \emph{$\varepsilon$-homogeneous} if $d(A,B) \in [0,\varepsilon] \cup [1-\varepsilon,1]$. It is easy to verify that homogeneity is a strictly stronger notion than regularity (up to a polynomial change in $\varepsilon$). Homogeneity for $k$-partite $k$-graphs is defined analogously; namely, a $k$-tuple of parts $A_1,\dots,A_k$ is \emph{$\varepsilon$-homogeneous} if $d(A_1,\dots,A_k) \in [0,\varepsilon] \cup [1-\varepsilon,1]$.
For convenience, we focus on the $k$-partite setting, but we remark that this restriction is not essential and can be easily removed (see \cref{rem:partite to general} for details).  
Thus, for a $k$-partite $k$-graph $H$ with parts $A_1,\dots,A_k$, we only consider vertex partitions which respect the partition into $A_1,\dots,A_k$.
Such a partition, consisting of a partition $\P_i$ of $A_i$ for every $i \in [k]$, is said to be \emph{$\varepsilon$-homogeneous} if the total sum of $|X_1|\dotsb|X_k|$ over all $k$-tuples $(X_1,\dots,X_k) \in \P_1 \times \dots \times \P_k$ which are not $\varepsilon$-homogeneous is at most $\varepsilon |A_1| \dotsb |A_k|$.

Note that it is no longer the case that all graphs or $k$-graphs have an $\varepsilon$-homogeneous partition of bounded size; for example, a random $k$-partite $k$-graph of edge density $\frac 12$ does not have a $\frac 14$-homogeneous partition into any bounded number of parts. However, our second main result demonstrates that if every link has a small homogeneous partition, then so does the hypergraph itself. 
Before stating this result precisely, we recall some of the history and known results on the question of which graphs and hypergraphs admit small homogeneous partitions.

In the case of graphs, the answer has been well-understood for some time, and it turns out that the key notion is that of having \emph{bounded VC-dimension}. 
We recall that the VC-dimension of a graph $G$ is the largest integer $d$ such that there exist distinct vertices $v_1,\dots,v_d \in V(G)$ with the property that for all $S \subseteq [d]$, there exists a vertex $w_S$ which is adjacent to $\{v_i : i \in S\}$ and non-adjacent to $\{v_j : j\notin S\}$. Generally, the precise VC-dimension of a graph is not very important, and we are only interested in the property of having \emph{bounded} VC-dimension, i.e.\ bounded by an absolute constant independent of the order of the graph. The property of having bounded VC-dimension turns out to be a deep and fundamental ``low-complexity'' notion for graphs, and a large number of works (e.g.\ \cite{FPS,MR4357431,1007.1670,2502.09576,MR3646875,MR4777869,2408.04165}) have established that certain difficult problems become much more tractable when restricted to graphs of bounded VC-dimension.

In particular, VC-dimension is intimately connected with regular and homogeneous partitions. Indeed, it is well-known \cite{MR2341924,FPS,MR2815610} that every graph with bounded VC-dimension has an $\varepsilon$-homogeneous (and hence $\poly(\varepsilon)$-regular) partition of size $(1/\varepsilon)^{O(1)}$. Moreover, this is an if and only if characterization in two distinct ways: first, a hereditary class of graphs admits homogeneous partitions of (any) bounded size if and only if it has bounded VC-dimension, which in turn happens if and only if it 
admits small $\varepsilon$-regular partitions (and otherwise tower-type bounds are required); see e.g.\ \cite{MR3943117,2404.01293} for details.

It is natural to ask for extensions of these characterizations to hypergraphs;
for example, which classes of hypergraphs admit polynomially-sized regular partitions or polynomially-sized homogeneous partitions? Moreover, as hypergraph regularity is substantially subtler and more involved than graph regularity, there are additional questions that arise, such as characterizing hypergraphs which admit full regularity partitions of sub-wowzer type. For more information on these questions, and for recent progress, see e.g.\ \cite{2111.01737,2404.01293,2404.02024}. In what follows, we focus on the existence and size of homogeneous partitions.

The first progress in this direction was due to Chernikov--Starchenko \cite{MR4350155} and to Fox--Pach--Suk \cite{FPS}, who independently introduced a strong notion of bounded VC-dimension for hypergraphs, and proved that such hypergraphs admit $\varepsilon$-homogeneous partitions of size $\poly(1/\varepsilon)$. More recently, Terry \cite[Theorem 1.16]{2404.01293} proved that this is an if and only if characterization in uniformity $3$: a class of $3$-graphs has bounded VC-dimension in this sense if and only if it admits polynomially-sized homogeneous partitions. However, in contrast to the graph case, the class of hypergraphs admitting homogeneous partitions of (some) bounded size is actually larger: it suffices for all links to have bounded VC-dimension. To state this precisely, we introduce the following definition \cite{2111.01737}.

\begin{definition}
    Let $k \geq 3$ and let $H$ be a $k$-graph. For distinct vertices $v_1,\dots,v_{k-2} \in V(H)$, their \emph{link} $L(v_1,\dots,v_{k-2})$ is the graph on 
$V(H) \setminus \{v_1,\dots,v_{k-2}\}$ with edge set 
$\{xy : xyv_1\dots v_{k-2} \in \nolinebreak E(H)\}$.
    The \emph{slicewise VC-dimension\footnote{In this context, ``slice'' is a synonym for ``link''. In \cite{2010.00726}, this quantity is called the \emph{VC$_1$-dimension} of $H$.}} of $H$ is defined as the maximum VC-dimension of a link $L(v_1,\dots,v_{k-2})$ over all $v_1,\dots,v_{k-2} \in V(H)$.
\end{definition}
That is, $H$ has bounded slicewise VC-dimension if and only if all of its links have bounded VC-dimension. 
Chernikov and Towsner \cite{2010.00726} proved that if $H$ has bounded slicewise VC-dimension, then $H$ admits an $\varepsilon$-homogeneous partition of bounded size (and in the case $k=3$, this result was independently proved by Terry and Wolf \cite{2111.01737}). Moreover, bounded slicewise VC-dimension turns out to be an if and only if characterization for when a hypergraph admits homogeneous partitions (see \cite[Theorem 2.34]{2111.01737} or \cite[Theorem 7.1]{2010.00726}) of any bounded size. The results of \cite{2010.00726,2111.01737}, however, give very weak quantitative information on the number of parts in the resulting homogeneous partition: Chernikov and Towsner use infinitary techniques that give no bound at all, and Terry and Wolf use the hypergraph regularity lemma, and thus obtain wowzer-type bounds on the number of parts in the homogeneous partition of $H$. 

These bounds were substantially improved by Terry, who gave a double-exponential upper bound \cite[Theorem 1.4]{2404.01274} in the case $k=3$. Moreover, Terry proved a single-exponential lower bound \cite[Theorem 6.8]{2404.01293}; more precisely, she constructed a $3$-graph with bounded slicewise VC-dimension which has no $\varepsilon$-homogeneous partitions with fewer than $2^{(1/\varepsilon)^{\Omega(1)}}$ parts. Terry writes that the main problem left open by her work is to close the gap between the single-exponential and double-exponential bounds, and she conjectured \cite{2404.01293,2404.01274} that the double-exponential upper bound is best possible. 

Our second main theorem disproves Terry's conjecture, showing that the single-exponential bound is the truth.
\begin{theorem}\label{thm:VC1}
    Let $H$ be a $k$-partite $k$-graph with bounded slicewise VC-dimension, and let $\varepsilon>0$. Then $H$ has an $\varepsilon$-homogeneous equipartition into $2^{(1/\varepsilon)^{O(1)}}$ parts.
\end{theorem}
In addition to providing optimal bounds for this problem, \cref{thm:VC1} also gives a short and much simpler proof of the result of \cite{2010.00726,2111.01737,2404.01274} that bounded slicewise VC-dimension implies the existence of homogeneous partitions.

In fact, we deduce \cref{thm:VC1} from the following result, which follows the theme introduced in \cref{thm:lower bound}; it states that if all links in $H$ have homogeneous partitions of size $r$, then $H$ itself has a homogeneous partition of size exponential in $r$. This immediately implies \cref{thm:VC1} since, as discussed above, all graphs of bounded VC-dimension have homogeneous partitions of polynomial size. Moreover, it nicely complements \cref{thm:lower bound}: knowing that all links have small regular partitions is essentially useless for finding a regular partition of $H$, but knowing that all links have small homogeneous partitions does yield a small homogeneous partition of $H$.

\begin{theorem}\label{thm:upper bound}
    Let $\varepsilon \in (0,\frac 12)$, let $H$ be a $k$-partite $k$-graph with all parts of size $n$, and suppose that for all $v_1,\dots,v_{k-2} \in V(H)$ in distinct parts, the bipartite\footnote{In the $k$-partite setting, when $v_1,\dots,v_{k-2}$ come from distinct parts, we again delete the isolated vertices and view $L(v_1,\dots,v_{k-2})$ as a bipartite graph between the two parts not containing any of $v_1,\dots,v_{k-2}$.} graph $L(v_1,\dots,v_{k-2})$ has an $\varepsilon'$-homogeneous partition of size at most $r$, where $\varepsilon' = c_k \varepsilon^6$ for a small enough constant $c_k > 0$ depending only on $k$. Then $H$ has an $\varepsilon$-homogeneous equipartition into at most $2^{(r/\varepsilon)^{O(1)}}$ parts.
\end{theorem}
\begin{remark}\label{rem:partite to general}
    In both \cref{thm:VC1,thm:upper bound},
    the statement for $k$-partite $k$-graphs immediately implies the corresponding result for general $k$-graphs.\footnote{Here, the definition of an $\varepsilon$-homogeneous partition of a (general) $k$-graph includes $k$-tuples $(X_1,\dots,X_k)$ where some of the $X_i$'s are equal; namely, the definition requires the sum of $|X_1|\dotsb|X_k|$ over all non-$\varepsilon$-homogeneous $k$-tuples of parts $(X_1,\dots,X_k)$ (including those which repeat parts) to be at most $\varepsilon |V(H)|^k$.}
    Indeed, given a $k$-graph $H$, we may pass to its $k$-partite cover $\widehat H$, whose vertex set is $k$ disjoint copies of $V(H)$ and whose edges are all transversal $k$-tuples corresponding to edges of $H$. It is easy to see that every link in $\widehat H$ is either empty or equal to the bipartite cover of a link in $H$, hence if all links in $H$ have small homogeneous equipartitions, then the same holds for $\widehat H$. We may then apply \cref{thm:upper bound} to $\widehat H$, obtaining $k$ different partitions of $V(H)$, one for each part of $V(\widehat H)$, each with at most $2^{(r/\varepsilon)^{O(1)}}$ parts. We now take the common refinement of these $k$ partitions to obtain a partition of $V(H)$ whose number of parts is $(2^{(r/\varepsilon)^{O(1)}})^k = 2^{(r/\varepsilon)^{O(1)}}$. Finally, it is not hard to check that if the original equipartition of $\widehat H$ was $\varepsilon$-homogeneous, then the resulting partition of $H$ is $\sqrt \varepsilon$-homogeneous. This is actually an instance of a more general fact (a simple consequence of Markov's inequality), namely that any refinement of an $\varepsilon$-homogeneous partition is itself $\sqrt \varepsilon$-homogeneous (see e.g.\ \cite[Lemma 2.1]{2404.01274}).
\end{remark}
\begin{remark}
    It is natural to ask for a version of \cref{thm:upper bound} where we assume that all links of $\ell$-tuples of vertices have small homogeneous partitions, for some $1 \leq \ell \leq k-2$. However, it is not hard to check that this assumption is weakest for $\ell=k-2$, in the following sense: 
    If in a $k$-partite $k$-graph, the link of every $\ell$-tuple has an $\varepsilon''$-homogeneous partition of size $r$, then all but $\varepsilon' n^{k-2}$ of the $(k-2)$-links have an $\varepsilon'$-homogeneous partition of size $r$, provided that $\varepsilon'' = \poly(\varepsilon')$ is small enough. Indeed, this partition is simply the $\varepsilon'$-homogeneous partition of the link of the first $\ell$ vertices in the $(k-2)$-tuple. The reason this works is that if a hypergraph has edge density very close to $0$ or $1$, then the same holds for the link of almost every vertex, so any given $\ell$-tuple extends to only few $(k-2)$-tuples for which the given partition is not $\varepsilon'$-homogeneous. Finally, our proof of \cref{thm:upper bound} tolerates a small number of $(k-2)$-tuples without an $\varepsilon'$-homogeneous partition of their link. 
    Indeed, one only needs to adapt the proof of Lemma \ref{lem:AxB partition}. To this end, one changes the definition of an $i$-good tuple $\mathbf{v} \cup \{x\}$ to require further that the link of $\mathbf{v}$ has an $\varepsilon'$-homogeneous partition. The rest of the proof carries through essentially as is (with slightly different numbers). 
    Summarizing, the analogue of Theorem \ref{thm:upper bound} holds for $\ell$-links, for any $1\leq \ell<k-2$. 
\end{remark}

\subsection{Proof overview}\label{sec:overview}
We now briefly sketch the proofs of our main theorems, starting with the upper bound, \cref{thm:upper bound}. For this high-level discussion, we restrict our attention to $3$-graphs, which already capture the heart of the problem. 

Thus, let $H$ be a $3$-partite $3$-graph with parts $A,B,C$, each of size $n$, and assume that the link of every vertex has a small $\varepsilon$-homogeneous equipartition, say of size $r = \poly(1/\varepsilon)$ (the case of general $r$ is no more complicated). It is not hard to show that the refinement of any homogeneous partition is again homogeneous (with only a polynomial loss in the parameters), so a natural approach is to simply take the common refinement of the homogeneous partitions of $L(c)$ over certain carefully chosen $c \in C$. This is essentially the approach taken by Terry \cite{2404.01274}, who first classifies the vertices in $c$ according to the ``structure'' of the homogeneous partition of $L(c)$. However, it seems very difficult to follow such an approach without incurring double-exponential bounds: Terry loses one exponential in this classification by structure, and another exponential from needing to take the common refinement of all of these partitions. The key new idea in our approach is to begin with an ``intermediate'' partition, which is \emph{not} a vertex partition, but rather a partition of $A \times B$ into a collection of (completely  unstructured) bipartite graphs. By working with such partitions for much of the proof, we are able to maintain polynomial dependencies almost throughout, and only pay a single exponential at the end to take a common refinement. As discussed above, Terry \cite{2404.01293} also proved a single-exponential lower bound for this problem, so this final step is in a sense unavoidable.

Namely, the first step is to partition $A \times B$ into bipartite graphs $E_1 \cup \dots \cup E_t$ with the following property: if $(a,b)$ and $(a',b')$ lie in the same class $E_i$ of this partition, then they have \emph{similar neighborhoods}, in the sense that the number of $c \in C$ such that $(a,b,c) \in E(H)$ but $(a',b',c) \notin E(H)$ (or vice versa) is at most $\varepsilon n$. 
To define this partition, let us first fix some $(a,b) \in A \times B$, and consider the homogeneous partition of $L(a)$. The vertex $b$ lies in some part, say $B_i$, of this partition. The key observation now is that for almost all $b' \in B_i$, the vertices $b$ and $b'$ have similar neighborhoods in the graph $L(a)$, simply because they lie in the same part of a homogeneous partition: almost all $c \in C$ have the same edge relation to both $b$ and $b'$. By the definition of the link $L(a)$, this immediately implies that $(a,b)$ and $(a,b')$ have similar neighborhoods in the sense above. As $|B_i| = n/r = \poly(\varepsilon)\cdot n$, we conclude that there are $\poly(\varepsilon)\cdot n$ choices of $b'$ such that $(a,b)$ and $(a,b')$ have similar neighborhoods. By repeating the same argument, but now looking at the link $L(b')$ and considering the part containing $a$, we find that there are also $\poly(\varepsilon)\cdot n$ choices of $a'$ such that $(a',b')$ and $(a,b')$ have similar neighborhoods. In total, we find $\poly(\varepsilon)\cdot n^2$ choices of $(a',b')$ which have a similar neighborhood to $(a,b)$. At this point, it is straightforward to use random sampling to find a partition of (almost all of) $A \times B$ into $t=\poly(1/\varepsilon)$ many bipartite graphs such that in each class (apart from a small exceptional class $E_0$), all pairs $(a,b),(a',b')$ have similar neighborhoods. 

As all pairs in a given part $E_i$ have nearly the same neighborhood in $C$, we can now partition $C$ into at most $2^t$ parts according to this information (and this is the only step where we pay an exponential). In this way, we obtain a homogeneous ``vertex-edge'' partition: we have partitioned $C$ and $A \times B$ into parts, such that for almost all choices of a part in $C$ and a part in $A \times B$, either almost all triples or almost none of the triples defined by these parts are edges of $H$. To convert this vertex-edge partition into a vertex partition of $A \cup B \cup C$, we simply repeat the argument above twice more, obtaining in turn partitions of $(A, B \times C)$ and of $(B, A \times C)$. Finally, it is not hard to show that the vertex partition of $A\cup B\cup C$ arising in this way is homogeneous, proving \cref{thm:upper bound}. 

We now turn to discussing the proof of the lower bound, \cref{thm:lower bound}. The key idea here is a simple observation about the structure of Gowers's \cite{Gowers} ingenious tower-type lower bound for Szemer\'edi's regularity lemma. Namely, Gowers's example can be viewed as an overlay of $t=\Theta(\log \frac{1}{\varepsilon})$ many bipartite graphs $G_1,\dots,G_t$, each of which is itself $\varepsilon$-regular (or, more precisely, has an $\varepsilon$-regular partition into a small number of parts). Nonetheless, these graphs interact in complicated ways, which causes their union to have no small $\varepsilon$-regular partition. 

In our proof of \cref{thm:lower bound}, we leverage this observation as follows. We place the exact same graphs $G_1,\dots,G_t$ on vertex set $A \cup B$. We also partition $C$ into $t$ parts $C_1,\dots,C_t$, and define a $3$-graph by declaring that the link of each $c \in C_i$ is precisely the graph $G_i$, for all $1 \leq i \leq t$. By construction, all links of vertices in $C$ have small $\varepsilon$-regular partitions. Moreover, this construction immediately implies that the link of every $a \in A$ is the union of $t$ complete bipartite graphs, namely between $N_{G_i}(a)$ and $C_i$, for all $i$; this structure readily yields a small $\varepsilon$-regular partition of $L(a)$, and the symmetric argument handles links of vertices in $B$. Finally, by modifying Gowers's analysis of his construction, we show that the hypergraph $H$ does not have any weakly $\varepsilon$-regular partitions into fewer than $\twr(\Omega(\log \frac{1}{\varepsilon}))$ parts, proving \cref{thm:lower bound}.

\paragraph{Paper organization:} We prove the upper bound \cref{thm:upper bound} in \cref{sec:upper}, and the lower bound \cref{thm:lower bound} in \cref{sec:lower}. We end in \cref{sec:conclusion} with some concluding remarks and three tantalizing open problems, related to analogues of the Erd\H os--Hajnal conjecture, R\"odl's theorem, and the induced counting lemma.
We omit floor and ceiling signs whenever these are not crucial. 

    \section{Proof of Theorem \ref{thm:upper bound}}\label{sec:upper}

    In this section we prove Theorem \ref{thm:upper bound}. Throughout this section, we assume for convenience that the number of vertices $n$ is divisible by quantities determined by the other parameters. All proofs work without this assumption with very minor changes. 
    
    We need the following simple lemma, showing that an $\varepsilon$-homogeneous partition of a bipartite graph translates to a partition of almost all vertices on one of the sides so that any two vertices in the same part have almost the same neighborhood. We also require all parts, except for a small exceptional set, to have the same size. Finally, for technical reasons it is important that we can specify the number of parts; namely, this number should be the same for all applications with the same parameters, rather than only being upper bounded by a function of these parameters.

    \begin{lemma}\label{lem:homogeneous equipartition}
        Let $\gamma > 0$ and $r \in \mathbb{N}$, and set $q \coloneqq \lceil (1-\gamma)\frac{3r}{\gamma} \rceil$. Let $G$ be a bipartite graph with parts $X,Y$ of size $n$ each, and suppose that $G$ has an $\gamma'$-homogeneous partition $X = X_1 \cup \dots \cup X_s$ and $Y = Y_1 \cup \dots \cup Y_t$, where $s \leq r$ and $\gamma' \coloneqq  \frac{\gamma^3}{48}$. Then there is a partition $X = X'_0 \cup X'_1 \cup \dots X'_q$ such \nolinebreak that
        \begin{enumerate}
            \item $|X'_1| = \dots = |X'_q|$ and $|X'_0| \leq \gamma n$;
            \item for every $i \in [q]$ and $x,x' \in X'_i$, it holds that $|N_Y(x) \triangle N_Y(x')| \leq \gamma n$.
        \end{enumerate}
    \end{lemma}
    \begin{proof}
        We say that $i \in [s]$ is {\em good} if the sum of $|Y_j|$ over $j \in [t]$ for which $(X_i,Y_j)$ is not $\gamma'$-homogeneous is at most $\frac{\gamma^2}{16} n$. Otherwise $X_i$ is {\em bad}. As the given partition is $\gamma'$-homogeneous, we get from Markov's inequality that 
        $\sum_{X_i \text{ bad}}|X_i| \leq \frac{\gamma' n^2}{\gamma^2 n/16} = \frac{\gamma}{3} n$.
        Place all bad sets $X_i$ into the exceptional set $X'_0$ (which we construct throughout the proof).

        Observe that if $(X_i,Y_j)$ is $\gamma'$-homogeneous, then the number of triples $x,x',y$ such that $x,x' \in X_i$ and $y \in N_Y(x) \triangle N_Y(x')$ is at most $\gamma' |X_i|^2|Y_j|$. Indeed, if e.g.~$d(X_i,Y_j) \geq 1-\gamma'$ (the other case is symmetric), then the number of triples $x,x',y$ containing a non-edge is at most $\gamma' |X_i|^2|Y_j|$.

        Consider a good set $X_i$. Let $\mathcal{T}$ be the set of triples $x,x',y$ such that $x,x' \in X_i$, $y \in Y$ and $y \in N_Y(x) \triangle N_Y(x')$. By the above and the choice of $\gamma'$, we have
        $
        |\mathcal{T}| \leq (\frac{\gamma^2}{16} + \gamma')|X_i|^2 n \leq 
        \frac{\gamma^2}{12}|X_i|^2 n. 
        $
        Hence, there is $x_i \in X_i$ participating in at most 
        $\frac{\gamma^2}{6} |X_i|n$ triples in $\mathcal{T}$. This means that the number of $x \in X_i$ with $|N_Y(x) \triangle N_Y(x_i)| \geq \frac{\gamma}{2}n$ is at most $\frac{\gamma}{3} |X_i|$. Move all of such elements $x$ to $X'_0$, and for every $i$, let $X_i''$ denote the remaining elements of $X_i$. We now have that $|X'_0| \leq \frac{2\gamma}{3}n$. Also, this gives a partition $X''_1,\dots,X''_s$ of $X \setminus X'_0$ such that for every $x \in X''_i$ it holds that 
        $|N_Y(x) \triangle N_Y(x_i)| \leq \frac{\gamma}{2}n$. By the triangle inequality,
        $|N_Y(x) \triangle N_Y(x')| \leq \gamma n$ for all $x,x' \in X''_i$. 

        Now set $m \coloneqq  \frac{\gamma n}{3r}$ and split each $X''_i$ into sets of size $m$ and possibly one leftover set of size less than $m$. Move all leftover sets to $X'_0$. Now $|X'_0| \leq \frac{2\gamma}{3}n + ms \leq \gamma n$, using $s \leq r$. This means that the number of parts of size $m$ disjoint from $X'_0$ is at least $\frac{(1-\gamma)n}{m} = (1-\gamma)\frac{3r}{\gamma}$, and hence at least $q$. 
        Move additional parts to $X'_0$ until the number of parts disjoint from $X'_0$ is exactly $q$. 
        This gives the desired partition $X'_0,X'_1,\dots,X'_q$. Note that $|X'_0| = n-qm \leq n - (1-\gamma)\frac{3r}{\gamma} \cdot m = \gamma n$, as required. 
    \end{proof}

    The next lemma is the key step of the proof, yielding a partition of $A_1 \times \dots \times A_{k-1}$ into $(k-1)$-partite $(k-1)$-graphs, with the property that all $(k-1)$-tuples in the same class have similar neighborhoods in $A_k$. This is, in a sense, a generalization of \cref{lem:homogeneous equipartition} to $k$-graphs, but, as discussed in \cref{sec:overview}, the main new insight is to partition $(k-1)$-tuples of vertices, rather than vertices, in this step of the argument.

	\begin{lemma}\label{lem:AxB partition}
		Let $\varepsilon > 0$, let $H$ be a $k$-partite $k$-graph with parts $A_1,\dots,A_k$ of size $n$ each, and suppose that for every $(k-2)$-tuple of vertices $v_1,\dots,v_{k-2}$ in distinct parts, $L(v_1,\dots,v_{k-2})$ has an $\varepsilon'$-homogeneous partition of size at most $r$, where $\varepsilon' = \frac{1}{48}(\frac{\varepsilon}{6k})^3$.
		Then there is a partition $A_1 \times \nolinebreak \dots \times \nolinebreak A_{k-1} = E_0 \cup E_1 \cup \dots \cup E_t$ with 
        $t \leq ( r/\varepsilon )^{O(1)}$ such that $|E_0| \leq \varepsilon n^{k-1}$ and such that for every $1 \leq i \leq t$ and every $e,e' \in E_i$, it holds that $|N_{A_k}(e) \triangle N_{A_k}(e')| \leq \varepsilon n$. 
	\end{lemma}

    \begin{proof}
        For each $i \in [k-1]$ and $(k-2)$-tuple $\mathbf{v} \in \prod_{j \in [k-1] \setminus \{i\}}A_j$, apply Lemma \ref{lem:homogeneous equipartition} to the link of $\mathbf{v}$ (which is a bipartite graph between $A_i$ and $A_k$) with parameter $\gamma \coloneqq  \frac{\varepsilon}{6k}$ to obtain a partition 
		$A_i = X_0^{(\mathbf{v})} \cup X_1^{(\mathbf{v})} \cup \dots \cup X_q^{(\mathbf{v})}$ such that
        $q = \lceil(1-\frac{\varepsilon}{6k})\frac{18kr}{\varepsilon}\rceil\leq \frac{18kr}{\varepsilon}$, $|X_0^{(\mathbf{v})}| \leq \gamma n$, $|X_1^{(\mathbf{v})}| = \dots = |X_q^{(\mathbf{v})}| \geq \frac{(1-\gamma)n}{q}$, and for every $1 \leq \ell \leq q$ and $x,x' \in X_\ell^{(\mathbf{v})}$ it holds that
		$|N_{A_k}(\mathbf{v} \cup \{x\}) \triangle N_{A_k}(\mathbf{v} \cup \{x'\})| \leq \gamma n$. 
        Note that due to the statement of Lemma \ref{lem:homogeneous equipartition}, the number of parts $q$ does not depend on $\mathbf{v}$.
        For $i \in [k-1]$, a tuple $e \in A_1 \times \dots \times A_{k-1}$ is called {\em $i$-bad} if, writing $e = \mathbf{v} \cup \{x\}$ for $\mathbf{v} \in \prod_{j \in [k-1] \setminus \{i\}}A_j$ and $x \in A_i$, we have $x \in X_0^{(\mathbf{v})}$. Otherwise $e$ is {\em $i$-good}. Note that there are at most $\gamma n^{k-1}$ $i$-bad tuples.

        Next,
	given two $i$-good $(k-1)$-tuples $e,e' \in A_1 \times \dots \times A_{k-1}$, we call them {\em $i$-twins} if they arise from the same $\mathbf v \in \prod_{j \in [k-1]\setminus \{i\}}A_j$ and the same set $X_\ell^{(\mathbf v)}$. More precisely, $e,e'$ are $i$-twins if, when writing $e = \mathbf v \cup \{x\}, e' = \mathbf v' \cup \{x'\}$, we have that $\mathbf v = \mathbf v'$ and $x,x'$ both lie in the same part $X_\ell^{(\mathbf v)}$, for some $1 \leq \ell \leq q$.
        Note that if $e,e'$ are $i$-twins then $|N_{A_k}(e) \triangle N_{A_k}(e')| \leq \gamma n$.
        Moreover, for each $i$-good tuple $e$, there are at least $\frac{(1-\gamma)n}{q}$ and at most $\frac{n}{q}$ tuples $e'$ such that $e,e'$ are $i$-twins.

        Additionally,
        let us say that two tuples $e, e_1 \in A_1 \times \dots \times A_{k-1}$ are \emph{chain twins} if there exists a sequence $e_1,e_2,\dots,e_k=e$ such that $e_i, e_{i+1}$ are $i$-twins for each $1 \leq i \leq k-1$. Note that the choice of $e_2,\dots,e_{k-1}$ is determined by $e$ and $e_1$, as $e_i$ is obtained from $e_1$ by changing all coordinates $1,\dots,i-1$ to be as in $e$. 
        Next, let us say that a tuple $e \in A_1 \times \dots \times A_{k-1}$ is \emph{excellent} if there are at least $(\frac{\gamma n}{q})^{k-1}$ choices of $e_1 \in A_1 \times \dots \times A_{k-1}$ such that $e$ and $e_1$ are chain twins. Our next claim shows that almost all tuples are excellent.

		\begin{claim}\label{claim:excellent}
        There are at least $(1-\frac \varepsilon 2)n^{k-1}$ excellent tuples $e \in A_1 \times \dots \times A_{k-1}$.
		\end{claim}
		\begin{proof}
        For the proof, we need to extend our definition of excellent tuples. 
        For $0 \leq i \leq k-1$, we say that $e \in A_1 \times \dots \times A_{k-1}$ is {\em $i$-excellent} if there are at least $(\frac{\gamma n}{q})^i$ choices of $e_1 \in A_1 \times \dots \times A_{k-1}$ for which there is a sequence
		$e_1,\dots,e_i,e_{i+1}=e$ where $e_j,e_{j+1}$ are $j$-twins for every $1 \leq j \leq i$. Note that, as above, such a sequence is uniquely determined by $e_1$ and the value of $i$. Moreover, in this terminology, an excellent tuple is the same as a $(k-1)$-excellent tuple. 

        In order to prove the claim, we will show by induction on $i$ that there are at least $(1-3i\gamma)n^{k-1}$ $i$-excellent tuples, for all $0 \leq i \leq k-1$. Note that the $i=k-1$ case suffices to prove the claim, as $1-3\gamma(k-1)\geq 1-\frac \varepsilon 2$. Moreover, the base case $i=0$ of the induction is trivial, so we now turn to the inductive step.
 
        As such, let $1 \leq i \leq k-1$, and let $\mathcal{F}$ be the set of the $(i-1)$-excellent tuples. By the induction hypothesis, 
			$
            |\mathcal{F}| \geq (1-3(i-1)\gamma)n^{k-1}.
            $
            Let $\mathcal{F}'$ be the set of all $e \in \mathcal{F}$ which are $i$-good. As there are in total at most $\gamma n^{k-1}$ $i$-bad tuples, we have that 
			$|\mathcal{F}'| \geq |\mathcal{F}| - \gamma n^{k-1} \geq (1-(3i-2)\gamma)n^{k-1}$. Every $e' \in \mathcal{F}'$ is $i$-good, hence has at least $\frac{(1-\gamma)n}{q}$ $i$-twins. Thus, the number of pairs $(e,e')$ where $e' \in \mathcal{F}'$ and $e,e'$ are $i$-twins is at least 
			$|\mathcal{F}'| \cdot \frac{(1-\gamma)n}{q} \geq (1-(3i-1)\gamma) \cdot \frac{n^k}{q}$.
            Let $\mathcal{G}$ be the set of all $e \in A_1 \times \dots \times A_{k-1}$ such that $e$ has at least $\gamma\frac{n}{q}$ $i$-twins $e' \in \mathcal{F}'$. 
            Using that each $e \in A_1 \times \dots \times A_{k-1}$ has at most $\frac{n}{q}$ $i$-twins, we get 
            $$
			|\mathcal{G}| \geq
			\frac{(1-(3i-1)\gamma) \cdot \frac{n^k}{q} - \frac{\gamma n^k}{q}}{\frac{n}{q}} = (1-3i\gamma)n^{k-1}.
			$$ 
            To complete the proof of the claim, we show that every $e \in \mathcal{G}$ is $i$-excellent. Indeed, fix any $e \in \mathcal{G}$. Fix any $e' \in \mathcal{F}'$ such that $e,e'$ are $i$-twins; there are at least $\frac{\gamma n}{q}$ choices for $e'$. As $e'$ is $(i-1)$-excellent, there are at least 
            $(\frac{\gamma n}{q})^{i-1}$ choices for $e_1,\dots,e_{i-1},e_i=e' \in A_1 \times \dots \times A_{k-1}$ such that $e_j,e_{j+1}$ are $j$-twins for every $1 \leq j \leq i-1$. Setting $e_{i+1} \coloneqq  e$, we get that $e_j,e_{j+1}$ are $j$-twins for every $1 \leq j \leq i$. The total number of choices for $e_1$ is at least $(\frac{\gamma n}{q})^i$, hence $e$ is indeed $i$-excellent. This completes the induction, and hence proves the claim by taking $i=k-1$.
		\end{proof}
        Now let $e$ be an excellent tuple, let $e_1$ be a chain twin of $e$, and let $e_2,\dots,e_{k-1}$ be such that,
        setting $e_k \coloneqq  e$, we have that $e_i,e_{i+1}$ are $i$-twins for every $1 \leq i \leq k-1$. By the triangle inequality,
		$$
        |N_{A_k}(e) \triangle N_{A_k}(e_1)| \leq \sum_{i=1}^{k-1} |N_{A_k}(e_{i+1}) \triangle N_{A_k}(e_i)| \leq (k-1)\gamma n \leq \frac{\varepsilon}{2}n.
        $$
        As $e$ is excellent, there are at least $(\frac{\gamma n}{q})^{k-1}$ choices for $e_1$ in the computation above. Therefore, there are at least
        $(\frac{\gamma n}{q})^{k-1}$ tuples $e' \in A_1 \times \dots \times A_{k-1}$ such that $|N_{A_k}(e) \triangle N_{A_k}(e')| \leq \frac{\varepsilon}{2}n$. 
		
		Now sample tuples $f_1,\dots,f_t \in A_1 \times \dots \times A_{k-1}$ uniformly at random and independently, where 
        $$
        t \coloneqq  \left( \frac{q}{\gamma} \right)^{k-1}\log\left(\frac{2}{\varepsilon} \right) = \left( \frac{r}{\varepsilon} \right)^{O(1)}.
        $$
        For $i \in [t]$, let $E_i$ be the set of $e \in A_1 \times \dots \times A_{k-1}$ such that $|N_{A_k}(e) \triangle N_{A_k}(f_i)| \leq \frac{\varepsilon}{2}n$ (if this holds for multiple $i$, we break ties arbitrarily). Then by the triangle inequality, $|N_{A_k}(e) \triangle N_{A_k}(e')| \leq \varepsilon n$ for every $e,e' \in E_i$. 

        Moreover, for any fixed $e \in A_1 \times \dots \times A_{k-1}$, we have that $e \in E_1 \cup \dots \cup E_t$ if some $f_i$ is a chain twin of $e$. In particular, if $e$ is excellent, then
		$$
        \mathbb{P}[e \notin E_1 \cup \dots \cup E_t] \leq \left( 1-\left(\frac\gamma q\right)^{k-1} \right)^t \leq e^{-(\frac{\gamma}{q})^{k-1} t} \leq \frac{\varepsilon}{2}.
        $$
        Moreover, by \cref{claim:excellent}, the number of excellent tuples is at least $(1-\frac \varepsilon2)n^{k-1}$. As a consequence, by linearity of expectation, we have that
        \[
        \mathbb E[|E_1 \cup \dots \cup E_t|] \geq \sum_{e\text{ excellent}} \mathbb P[e \in E_1 \cup \dots \cup E_t] \geq \sum_{e\text{ excellent}} \left(1-\frac \varepsilon 2\right) \geq \left(1-\frac \varepsilon 2\right)^2 n^{k-1} \geq (1-\varepsilon) n^{k-1}.
        \]
        Hence, there is an outcome with 
		$|E_1 \cup \dots \cup E_t| \geq (1-\varepsilon)n^{k-1}$, and setting 
		$E_0 \coloneqq  (A_1 \times \dots \times A_{k-1}) \setminus (E_1 \cup \dots \cup E_t)$ completes the proof of Lemma \ref{lem:AxB partition}. 
	\end{proof}
    
    The following lemma is a sort of converse to \cref{lem:homogeneous equipartition}, stating that if a partition is not homogeneous, then there are many ``witnesses'' which witness this by having a large symmetric difference of their neighborhoods. It is an easy consequence of \cite[Lemma 2.3]{FPS}, but we include a proof for completeness. 
	\begin{lemma}\label{lem:non-homogeneous partition}
		Let $H$ be a $k$-partite $k$-graph with parts $X_1,\dots,X_k$ of size $n$ each. For $i \in [k]$, let $\mathcal{P}_i$ be an equipartition of $X_i$ into $s$ parts, and suppose that $(\mathcal{P}_1,\dots,\mathcal{P}_k)$ is not $\varepsilon$-homogeneous. Then there are at least 
        $\varepsilon^2(1-\varepsilon)\frac{n^{k+1}}{s}$ pairs $e,e' \in X_1 \times \dots \times X_k$ such that $|e \cap e'| = k-1$, $e \in E(H), e' \notin E(H)$, and $e \triangle e'$ is contained in some part of $\mathcal{P}_i$ for some $i \in [k]$.
	\end{lemma}
	\begin{proof}
        Fix a $k$-tuple $(Y_1,\dots,Y_k) \in \mathcal{P}_1 \times \dots \times \mathcal{P}_k$ with 
		$\varepsilon \leq d(Y_1,\dots,Y_k) \leq 1-\varepsilon$. We claim that there are at least $\varepsilon(1-\varepsilon)(\frac{n}{s})^{k+1}$ pairs $e,e' \in Y_1 \times \dots \times Y_k$ with $|e \cap e'| = k-1$ and $e \in E(H), e' \notin E(H)$. This suffices because the number of non-$\varepsilon$-homogeneous $k$-tuples $(Y_1,\dots,Y_k)$ is at least $\varepsilon s^k$.

		To prove the above claim, sample vertices $u_i,v_i \in Y_i$ uniformly at random and independently, for each $i \in [k]$. For $0 \leq i \leq k$, let $e_i \coloneqq  (v_1,\dots,v_i,u_{i+1},\dots,u_k)$. So $e_0 = (u_1,\dots,u_k)$ and $e_k = (v_1,\dots,v_k)$. 
        Let us say that two $k$-tuples $e,e'$ {\em disagree} if one of them is in $E(H)$ while the other is not. As the random tuples $e_0,e_k$ are independent of one another,
        the probability that $e_0,e_k$ disagree is at least $2\varepsilon(1-\varepsilon)$, because $\varepsilon \leq d(Y_1,\dots,Y_k) \leq 1-\varepsilon$. If this happens, then there is $1 \leq i \leq k$ such $e_{i-1},e_i$ disagree. Note that $e_{i-1},e_i$ differ only in the $i$th coordinate. Let $t_i$ be the number of pairs $e,e' \in Y_1 \times \dots \times Y_k$ such that $e,e'$ differ only in the $i$th coordinate and $e \in E(H), e' \notin E(H)$. The probability that there exists $1 \leq i \leq k$ such that $e_{i-1},e_i$ disagree is at most 
        $$
        \sum_{i=1}^k \frac{2t_i}{|Y_1|\dotsb|Y_k|\cdot |Y_i|} = \sum_{i=1}^k \frac{2t_i}{(n/s)^{k+1}} \; .
        $$
        On the other hand, as explained above, this probability is at least $2\varepsilon(1-\varepsilon)$. It follows that $\sum_{i=1}^k t_i \geq \varepsilon(1-\varepsilon)(\frac{n}{s})^{k+1}$, proving our claim.
	\end{proof}
	\noindent We now have all the tools needed to prove \cref{thm:upper bound}.
	\begin{proof}[Proof of Theorem \ref{thm:upper bound}]
		We apply Lemma \ref{lem:AxB partition} $k$ times with parameter $\frac{\varepsilon^2}{8k}$, each time with a different set $X_i$ ($1 \leq i \leq k$) playing the role of $A_k$ (with the other $k-1$ sets $X_j$ playing the roles of $A_1,\dots,A_{k-1}$). This gives, 
		for every $i \in [k]$, a partition 
		$\prod_{j \in [k] \setminus \{i\}} X_j = E_0^{(i)} \cup E_1^{(i)} \cup \dots \cup E_t^{(i)}$ such that $|E_0^{(i)}| \leq \frac{\varepsilon^2}{8k} n^{k-1}$, and such that for every $1 \leq j \leq t$ and $e,e' \in E_j^{(i)}$ it holds that 
		$|N_{X_i}(e) \triangle N_{X_i}(e')| \leq \frac{\varepsilon^2}{8k} n$. Also,
		$t \leq (r/\varepsilon)^{O(1)}$. 
		
		For $i \in [k]$, we define a partition $\mathcal{P}_i$ of $X_i$ as follows. For each $j \in [t]$, fix $e^{(i)}_j \in E_j^{(i)}$ and put $X_j^{(i)} \coloneqq N_{X_i}(e_j^{(i)})$. Let $\mathcal{P}'_i$ be the common refinement of the sets $X_1^{(i)},\dots,X_t^{(i)}$. 
		Note that $|\mathcal{P}'_i| \leq 2^t \eqqcolon p$.
		Next, we refine $\mathcal{P}'_i$ further to make the parts have the same size. To this end, partition each part of $\mathcal{P}'_i$ into parts of size $m \coloneqq \frac{\varepsilon^2 n}{8kp}$ and (possibly) an additional leftover part of size less than $m$. Let $X_0^{(i)}$ be the union of the leftover parts; so $|X_0^{(i)}| \leq \frac{\varepsilon^2}{8k} n$. Partition\footnote{Here we use our assumption that $n$ is divisible by quantities determined by the other parameters, and in particular is divisible by $m$.} $X_0^{(i)}$ into parts of size $m$. 
        The resulting partition is $\mathcal{P}_i$. Note that $\mathcal{P}_i$ is an equipartition of size $s \coloneqq \frac{n}{m} = \frac{8kp}{\varepsilon^2} \leq 2^{(r/\varepsilon)^{O(1)}}$. 
		
		We claim that $(\mathcal{P}_1,\dots,\mathcal{P}_k)$ is an $\varepsilon$-homogeneous partition of $H$. For the sake of contradiction, suppose otherwise. 
		For $i \in [k]$, let $T_i$ be the number of pairs of $k$-tuples $e,e' \in X_1 \times \dots \times X_k$ such that 
		there is 
		$(X'_1,\dots,X'_k) \in \mathcal{P}_1 \times \dots \times \mathcal{P}_k$ with $e,e' \in X'_1 \times \dots \times X'_k$, 
		$|e \cap e'| = k-1$, $e \triangle e' \subseteq X'_i$, $e \in E(H)$ and $e' \notin E(H)$. 
		As we assumed that $(\mathcal{P}_1,\dots,\mathcal{P}_k)$ is not $\varepsilon$-homogeneous, Lemma \ref{lem:non-homogeneous partition} gives 
		$T_1 + \dots + T_k \geq \varepsilon^2(1-\varepsilon)\frac{n^{k+1}}{s}$. To get a contradiction, we now upper-bound each $T_i$. By symmetry, it suffices to consider $T_k$. Thus, we consider pairs $e,e'$ with 
		$e,e' \in X'_1 \times \dots \times X'_k$ for some 
		$(X'_1,\dots,X'_k) \in \mathcal{P}_1 \times \dots \times \mathcal{P}_k$, 
		$|e \cap e'| = k-1$, $e \triangle e' \subseteq X'_k$, $e \in E(H)$ and $e' \notin E(H)$. 
		For $i \in [k-1]$, let $x_i$ be the vertex of $e,e'$ in $X_i$. 
		The number of choices for $e,e'$ where $(x_1,\dots,x_{k-1}) \in E_0^{(k)}$ is at most $|E_0^{(k)}| \cdot s \cdot (\frac{n}{s})^2 \leq \frac{\varepsilon^2}{8k} \cdot \frac{n^{k+1}}{s}$. 
		Let now $j \in [t]$, and let us consider the case that $(x_1,\dots,x_{k-1}) \in E_j^{(k)}$. Recall that $|N_{X_k}(x_1,\dots,x_{k-1}) \triangle X_j^{(k)}| \leq \frac{\varepsilon^2}{8k}n$, where $X_j^{(k)} = N_{X_k}(e^{(k)}_j)$ as above. This holds because $(x_1,\dots,x_{k-1})$ and $e^{(k)}_j$ both belong to $E_j^{(k)}$, and as $|N_{X_k}(f) \triangle N_{X_k}(f')| \leq \frac{\varepsilon^2}{8k}n$ for any two $f,f' \in E_j^{(k)}$.
		Also, by the definition of 
		$\mathcal{P}_k$, each part of $\mathcal{P}_k$ is contained either in $X_0^{(k)}$, or in $X_j^{(k)}$, or in $X_k \setminus X_j^{(k)}$ (because $\mathcal{P}'_k$ is the common refinement of $X_1^{(k)},\dots,X_t^{(k)}$). We handle these cases separately.
		First, as $|X_0^{(k)}| \leq \frac{\varepsilon^2}{8k} n$, the number of choices for $e,e'$ where $e \triangle e' \subseteq X_0^{(k)}$ is at most 
		$n^{k-1} \cdot \frac{\varepsilon^2}{8k} n \cdot \frac{n}{s} = \frac{\varepsilon^2}{8k} \cdot \frac{n^{k+1}}{s}$. Now consider the case that $e \triangle e'$ is contained in $X_j^{(k)}$ or in $X_k \setminus X_j^{(k)}$. As $|N_{X_k}(x_1,\dots,x_{k-1}) \triangle X_j^{(k)}| \leq \frac{\varepsilon^2}{8k}n$, the number of vertices $y' \in X_j^{(k)}$ with $(x_1,\dots,x_{k-1},y') \notin E(H)$ is at most $\frac{\varepsilon^2}{8k} n$. Hence, the number of choices for $e,e'$ where $e \triangle e' \subseteq X_j^{(k)}$ is at most 
		$|E_j^{(k)}| \cdot \frac{\varepsilon^2}{8k} n \cdot \frac{n}{s} = 
		|E_j^{(k)}| \cdot \frac{\varepsilon^2}{8k} \cdot \frac{n^2}{s}$. Similarly, the number of vertices $y \in X_k \setminus X_j^{(k)}$ with $(x_1,\dots,x_{k-1},y) \in E(H)$ is at most $\frac{\varepsilon^2}{8k} n$. Hence, the number of choices for $e,e'$ where $e \triangle e' \subseteq X_k \setminus X_j^{(k)}$ is at most $|E_j^{(k)}| \cdot \frac{\varepsilon^2}{8k} \cdot \frac{n^2}{s}$. Collecting all terms, we get that
		$$
		T_k \leq 2 \cdot \frac{\varepsilon^2}{8k} \cdot \frac{n^{k+1}}{s} + 2 \cdot \frac{\varepsilon^2}{8k} \cdot \frac{n^2}{s}\sum_{j=1}^t |E_j^{(k)}| \leq 
		\frac{\varepsilon^2}{2k} \cdot \frac{n^{k+1}}{s}.
		$$ 
		By symmetry, $T_1,\dots,T_{k-1} \leq \frac{\varepsilon^2}{2k} \cdot \frac{n^{k+1}}{s}$ as well, so 
		$T_1 + \dots + T_k < \varepsilon^2(1-\varepsilon)\frac{n^{k+1}}{s}$, a contradiction. 
	\end{proof}

    \section{Proof of Theorem \ref{thm:lower bound}}\label{sec:lower}

    Throughout this section, we fix $0 < \delta \leq \varepsilon$ with $\varepsilon$ smaller than some implicit absolute constant, and also assume wherever needed that $n \geq n_0(\delta)$ is large enough.
    We will use the following lemma, which essentially appears in \cite{Gowers}.

    \begin{lemma}\label{lem:orthogonal partitions}
    Let $m,M$ be integers with $M \leq \max(e^{m/16},2)$. Then there are partitions $(X_i,Y_i)_{i=1,\dots,m}$ of $[M]$ having the following properties:
    \begin{enumerate}
        \item If $M \geq \log^3(8m^2)$, then $|X_i|,|Y_i| = \frac{M}{2} \pm M^{2/3}$ for every $1 \leq i \leq m$, and $|X_i \cap X_{i'}|, |X_i \cap Y_{i'}|, |Y_i \cap Y_{i'}| = \frac{M}{4} \pm M^{2/3}$ for every pair $1 \leq i \neq i' \leq m$.
        \item 
        Let $\zeta \leq 1/2$ and $\eta,\varepsilon > 0$ such that $(1-\eta)(1-4\varepsilon) \geq 1-\zeta + \zeta^2$. 
        Let $\lambda_1,\dots,\lambda_M \in \mathbb{R}_{\geq 0}$ with $\sum_{i=1}^M \lambda_i = 1$, and suppose that $\lambda_i \leq 1-\zeta$ for every $i \in [M]$. Then there are at least $\eta  m$ indices $i \in [m]$ satisfying $\min\left( \sum_{j \in X_i}\lambda_j, \sum_{j \in Y_i} \lambda_j \right) > \varepsilon$.  
    \end{enumerate}
    \end{lemma}
    
    Item 2 of Lemma \ref{lem:orthogonal partitions} follows immediately by combining Lemmas 3 and 5 in \cite{Gowers}. The desired partitions $(X_i,Y_i)$ are chosen uniformly at random, and Item 1 is a simple property of random partitions. For completeness, we include the proof of Lemma \ref{lem:orthogonal partitions} in \cref{sec:appendix}. 

    We now proceed with the proof of Theorem \ref{thm:lower bound}. We describe the construction used to establish the theorem in the following Section \ref{sec:construction}, and then analyze it in Section \ref{sec:analysis}. 
    As in \cite{Gowers}, it is convenient to describe the construction as a weighted 3-graph. We will then transform this into an unweighted (``normal") 3-graph via sampling. 

    \subsection{The construction}\label{sec:construction}
    We construct a weighted 3-partite 3-graph $H$ with parts $A,B,C$, each of size $n$. We set 
    $t \coloneqq  \lfloor \frac{1}{4}\log_7(\frac{1}{\varepsilon}) - 3 \rfloor$ and $s_0=\lceil 4/\delta^4 \rceil$. Partition $C$ into $t$ equal parts $C_1,\dots,C_t$. 

    For an integer $m$, let 
    $\phi(m) \coloneqq  \max\left( \lfloor e^{m/16} \rfloor ,2 \right)$, and note that the conclusion of Lemma \ref{lem:orthogonal partitions} holds for each $M \leq \phi(m)$. 
    We now inductively define a sequence of integers $m_0,m_1,\dots,m_t$ as follows. Set $m_0 = 1$.
    For each $1 \leq r \leq t$, if $m_{r-1} < s_0$ and $\phi(m_{r-1}) \geq s_0$ then set $m_r \coloneqq  m_{r-1} \cdot s_0$. Otherwise, set $m_r \coloneqq  m_{r-1} \cdot \phi(m_{r-1})$. Note that $m_{r-1}$ divides $m_r$ and $m_r > m_{r-1}$ (using that $\phi(m_{r-1}), s_0 \geq 2$). Also, $\frac{m_r}{m_{r-1}} \leq \phi(m_{r-1})$. To explain this somewhat unusual choice of a sequence, let us remark that we would generally like $m_r=m_{r-1}\cdot \phi(m_{r-1})$, as this enables us to apply \cref{lem:orthogonal partitions} while also ensuring that $m_r$ is exponentially larger than $m_{r-1}$, so as to obtain a tower-type bound in the end. However, if we have the sequence grow exponentially at every step, we may accidentally end up with a vertex whose link only has a $\delta$-regular partition of size exponential in $\delta$. In order to ensure that this doesn't happen, we make sure to stunt the growth of the sequence at the critical step (around the value $s_0$), thus ensuring a polynomial-sized $\delta$-regular partition in all cases. See the proof of \cref{lem:regular partitions in links} for the precise details.

    Next, let us identify each of $A,B$ with $[n]$ in order to have an order on these sets. For $0 \leq r \leq t$, let $\mathcal{A}_r$ be the partition of $A$ into $m_r$ equal-sized intervals (with respect to the vertex order), and similarly let $\mathcal{B}_r$ be the partition of $B$ into $m_r$ equal-sized intervals. For each $1 \leq r \leq t$, $\mathcal{A}_r$ refines $\mathcal{A}_{r-1}$ because $m_{r-1}$ divides $m_r$, and similarly for $\mathcal{B}_r$ and $\mathcal{B}_{r-1}$. 

    Now we define 3-partite 3-graphs $H_1,\dots,H_t \subseteq A \times B \times C$. Fix $1 \leq r \leq t$ and let us write $\mathcal{A}_{r-1} = \{A_1,\dots,A_m\}$ and $\mathcal{B}_{r-1} = \{B_1,\dots,B_m\}$, where $m \coloneqq  m_{r-1}$. Put also $M \coloneqq  \frac{m_r}{m_{r-1}} \leq \phi(m_{r-1}) = \phi(m)$. For each $i \in [m]$, let $A_{i,1},\dots,A_{i,M}$ (resp.\ $B_{i,1},\dots,B_{i,M}$) be the parts of $\mathcal{A}_r$ (resp.\ $\mathcal{B}_r$) contained in $A_i$ (resp.\ $B_i$). Let $(X_i,Y_i)_{i=1}^m$ be the partitions of $[M]$ given by Lemma \ref{lem:orthogonal partitions} (it is possible to invoke Lemma \ref{lem:orthogonal partitions} because $M \leq \phi(m)$). 
    Define a bipartite graph $G_r \subseteq A \times B$ as follows: 
    For each pair $1 \leq i,j \leq m$, put into $G_r$ all edges in 
    $\left( \bigcup_{k \in X_j} A_{i,k} \right) \times 
    \left( \bigcup_{k \in X_i} B_{j,k} \right)$ and 
    $\left( \bigcup_{k \in Y_j} A_{i,k} \right) \times 
    \left( \bigcup_{k \in Y_i} B_{j,k} \right)$. Now let $H_r = \{ e \cup \{v\} : e \in E(G_r), v \in C_r\} = G_r \times C_r$. 
    This completes the definition of $H_r$.
    Finally, let $H$ be the weighted 3-graph 
    $H \coloneqq \sum_{r=1}^t 2^{-r} H_r$. Note that, as $\sum_{r=1}^t 2^{-r}\leq 1$, every edge of $H$ has weight in $[0,1]$.
    
    For most of the proof, we shall work with the weighted 3-graph $H$. Consequently, we need to introduce the weighted analogues of the basic notions of link, density, and regularity. For $v \in A$, the link $L_H(v)$ is defined as the weighted graph on $B \times C$ where $e \in B \times C$ has weight $H(e \cup \{v\})$. The link of a vertex in $B$ or $C$ is defined analogously. For $A' \subseteq A, B' \subseteq B, C' \subseteq C$, the density of $(A',B',C')$ is defined as $d_H(A',B',C') = \frac{1}{|A'||B'||C'|}\sum_{e \in A' \times B' \times C'} H(e)$. Weak regularity for weighted 3-graphs is defined analogously to the unweighted case. 

    \subsection{The analysis}\label{sec:analysis}
    Throughout this section, we use the same notation as in Section \ref{sec:construction}. We begin with the following lemma, showing that if $m_{r-1}$ is large enough, then the graph $G_r$ is $\delta$-regular.

    \begin{lemma}\label{lem:large partitions are quasirandom}
        Let $1 \leq r \leq t$ such that $m_{r-1} \geq s_0$. Then the graph $G_r$ is $\delta$-regular. 
    \end{lemma}
    \noindent
    For the proof of Lemma \ref{lem:large partitions are quasirandom}, we need the following fact from \cite{ADLRY}.
    \begin{lemma}[{\cite[Lemma 3.2]{ADLRY}}]\label{lem:Chung-Graham-Wilson}
        Let $G$ be a bipartite graph with parts $A,B$ of size $n$ each. Let $d$ be the density of $G$ and let $2n^{-1/4} < \delta < \frac{1}{16}$. Suppose that
        \begin{enumerate}
            \item For all but at most $\frac{1}{8}\delta^4 n$ of the vertices $x \in B$, it holds that $|d_G(x) - dn| \leq \delta^4 n$;
            \item For every $B' \subseteq B$ with $|B'| \geq \delta n$, it holds that 
            $\sum_{x,y \in B'} \left( |N_G(x) \cap N_G(y)| - d^2n \right) < \frac{\delta^3}{2}n|B'|^2$.
        \end{enumerate}
        Then $G$ is $\delta$-regular. 
    \end{lemma}
    \begin{proof}[Proof of Lemma \ref{lem:large partitions are quasirandom}]
       We will show that Conditions 1--2 in Lemma \ref{lem:Chung-Graham-Wilson} are satisfied. 
        As in Section \ref{sec:construction}, write $m \coloneqq  m_{r-1}$, $M \coloneqq  \frac{m_r}{m_{r-1}}$, $\mathcal{A}_{r-1} = \{A_1,\dots,A_m\}$ and $\mathcal{B}_{r-1} = \{B_1,\dots,B_m\}$, and for each $i \in [m]$ let $A_{i_1},\dots,A_{i,M}$ (resp.\ $B_{i,1},\dots,B_{i,M}$) be the parts of $\mathcal{A}_r$ (resp.\ $\mathcal{B}_r$) contained in $A_i$ (resp.\ $B_i$). Let also $(X_i,Y_i)_{i=1}^m$ be the partitions of $[M]$ given by Lemma \ref{lem:orthogonal partitions}. Note that by the assumption of Lemma \ref{lem:large partitions are quasirandom}, we have $m \geq s_0$ and therefore $M = \phi(m) = \lfloor e^{m/16} \rfloor \geq e^{\Omega(s_0)}$. This also implies that $M \geq \log^3(8m^2)$ (as we assumed that $\varepsilon$ is at most some absolute constant, hence $s_0$ is at least some large constant), so we may apply Item 1 of \cref{lem:orthogonal partitions}.

        First, fix any $x \in B$ and let us consider the degree of $x$ in $G_r$. 
        Let $j \in [m]$ such that $x \in B_j$. 
        Fix any $1 \leq i \leq m$ and consider the edges between $x$ and $A_i$. 
        For convenience, set 
        $A_i^{(1)} = \bigcup_{k \in X_j} A_{i,k}$, 
        $A_i^{(2)} = \bigcup_{k \in Y_j} A_{i,k}$,
        $B_j^{(1)} = \bigcup_{k \in X_i} B_{j,k}$,
        $B_j^{(2)} = \bigcup_{k \in Y_i} B_{j,k}$.
        By definition, $G_r$ has all edges in 
        $A_i^{(1)} \times B_j^{(1)}$ and 
        $A_i^{(2)} \times B_j^{(2)}$ and no other edges between $A_i$ and $B_j$. By Item 1 of Lemma \ref{lem:orthogonal partitions}, $|X_i|,|Y_i|,|X_j|,|Y_j| = \frac{M}{2} \pm M^{2/3}$. This means that each of the sets 
        $A_i^{(1)},A_i^{(2)}$ has size 
        $(\frac{1}{2} \pm M^{-1/3}) |A_i|$.
        Therefore, the number of edges between $x$ and $A_i$ is 
        $\left( \frac{1}{2} \pm M^{-1/3} \right)|A_i|$. As this is true for every $1 \leq i \leq m$, it follows that $d_{G_r}(x) = (\frac{1}{2} \pm M^{-1/3})n = (\frac{1}{2} \pm \frac{\delta^4}{2})n$, using that $M \geq e^{\Omega(s_0)}$. It follows that the density $d$ of $G_r$ is $\frac{1}{2} \pm \frac{\delta^4}{2}$, and that for every $x \in B$ we have $|d_G(x) - dn| \leq \delta^4 n$.  

        Now we consider Item 2 in Lemma \ref{lem:Chung-Graham-Wilson}. 
        Fix any $B' \subseteq B$ with $|B'| \geq \delta n$. For convenience, put $f(x,y) \coloneqq |N_{G_r}(x) \cap N_{G_r}(y)| - d^2 n$ for $x,y \in B'$, and let
        $S \coloneqq \sum_{x,y \in B'}f(x,y)$.
        The pairs $x,y$ belonging to the same part of $\mathcal{B}_{r-1}$ contribute at most $|B'| \cdot \frac{n}{m} \cdot n = |B'|\frac{n^2}{m} \leq \frac{\delta^3}{4}n|B'|^2$ to $S$, using that $m \geq s_0$ and $|B'| \geq \delta n$. 
        Now consider the pairs $x,y$ belonging to different parts of $\mathcal{B}_{r-1}$. So fix $1 \leq j \neq j' \leq m$ such that $x \in B_j, y \in B_{j'}$. Fix any $1 \leq i \leq m$. As before, put $B_j^{(1)} = \bigcup_{k \in X_i} B_{j,k}$, 
        $B_j^{(2)} = \bigcup_{k \in Y_i} B_{j,k}$
        and similarly 
        $B_{j'}^{(1)} = \bigcup_{k \in X_{i}} B_{j',k}$, 
        $B_{j'}^{(2)} = \bigcup_{k \in Y_{i}} B_{j',k}$.
        Let $a,a' \in \{1,2\}$ such that $x \in B_j^{(a)}$ and 
        $y \in B_{j'}^{(a')}$. By the definition of $G_r$, the common neighborhood of $x,y$ in $A_i$ is 
        $\bigcup_{k \in Z}A_{i,k}$, \nolinebreak where
        $$
        Z \coloneqq  
        \begin{cases}
            X_i \cap X_{i'} & a=1, a'=1, \\
            X_i \cap Y_{i'} & a=1, a'=2, \\
            Y_i \cap X_{i'} & a=2, a'=1, \\
            Y_i \cap Y_{i'} & a=2, a'=2. \\
        \end{cases}
        $$
        By Item 1 of Lemma \ref{lem:orthogonal partitions}, in all cases we have $|Z| = \frac{M}{4} \pm M^{2/3}$. This means that the common neighborhood of $x,y$ in $A_i$ has size $(\frac{1}{4} \pm M^{-1/3})|A_i|$. Summing over all $1 \leq i \leq m$, we get that $|N_{G_r}(x) \cap N_{G_r}(y)| \leq (\frac{1}{4} + M^{-1/3})n$ and therefore
        $f(x,y) \leq (\frac{1}{4} + M^{-1/3})n - (\frac{1}{2} - M^{-1/3})^2 n \leq \frac{\delta^3}{4}n$, using that $M \geq e^{\Omega(s_0)}$. 
        It follows that
        $
        S \leq \frac{\delta^3}{4}n |B'|^2 + \frac{\delta^3}{4}n |B'|^2 = 
        \frac{\delta^3}{2}n |B'|^2.
        $
        Hence, both items of Lemma \ref{lem:Chung-Graham-Wilson} hold, implying that $G_r$ is $\delta$-regular.
    \end{proof}
    \noindent Given \cref{lem:large partitions are quasirandom}, it is now straightforward to show that all links in $H$ have small $\delta$-regular partitions.
    \begin{lemma}\label{lem:regular partitions in links}
    For every $v \in V(H)$, the link $L_H(v)$ of $v$ has a $\delta$-regular partition of size $O(\delta^{-8})$.
    \end{lemma}
    \begin{proof}
        Suppose first that $v \in C$, and let $1 \leq r \leq t$ such that $v \in C_r$. 
        Note that $L_H(v) = 2^{-r} \cdot G_r$.
        We consider two cases. 
        If $m_{r-1} \geq s_0$ then by Lemma \ref{lem:large partitions are quasirandom}, $G_r$, and hence also $L_H(v) = 2^{-r} \cdot G_r$, is $\delta$-regular, meaning that $(A,B)$ is a (trivial) $\delta$-regular partition of $L_H(v)$.
        Now suppose that $m_{r-1} < s_0$. 
        Then, by the definition of the sequence $m_0,\dots,m_t$, we have $m_r \leq m_{r-1} \cdot s_0 \leq s_0^2$ (it is here that we crucially use the definition of the sequence, making sure that $m_r$ is not exponentially larger than $s_0$). 
        Thus, the partitions $\mathcal{A}_r,\mathcal{B}_r$ have size at most $s_0^2 \leq 17/\delta^8$. Observe that for each $A' \in \mathcal{A}_r$, $B' \in \mathcal{B}_r$, $(A',B',C_r)$ is a complete or empty 3-partite 3-graph. This means that in $H$, either all edges in $(A',B',C_r)$ have weight $2^{-r}$, or all such edges have weight 0. Hence, $(\mathcal{A}_r,\mathcal{B}_r)$ is a $\delta$-regular partition of $L_H(v)$. 

        Next, suppose that $v \in A \cup B$. By symmetry, it suffices to handle the case that $v \in A$. For $1 \leq r \leq t$, let $N_r \subseteq B$ denote the neighborhood of $v$ in the graph $G_r$. 
        By the definition of $H_r$ and $H$, all edges in $N_r \times C_r$ have weight $2^{-r}$ in $L_H(v)$, while all edges in $(B \setminus N_r) \times C_r$ have weight 0 in $L_H(v)$.  
        Let $B_1,\dots,B_s$ be the atoms of the Venn diagram of $N_1,\dots,N_t$, and note that $s \leq 2^t \leq \frac{1}{\varepsilon}\leq \frac 1 \delta$. Then for each $1 \leq i \leq s$ and $1 \leq r \leq t$, all edges in $(B_i,C_r)$ have the same weight in $L_H(v)$. It follows that $(\{B_1,\dots,B_s\}, \{C_1,\dots,C_t\})$ is a $\delta$-regular partition of $L_H(v)$, as required. 
    \end{proof}

    The final step is showing that $H$ has no weakly $\varepsilon$-regular partitions with fewer than a tower-type number of parts. This step also follows the work of Gowers, so we recall some notation and ideas from \cite{Gowers}. For two sets $P,A$ and a real number $\beta$, we write $P \subset_\beta A$ if $\ab{P \cap A} \geq (1-\beta) \ab P$, or equivalently if $\ab{P \setminus A} \leq \beta \ab P$. For $\beta=0$, this is simply the usual subset relation. If $P_1,\dots,P_k$ and $A_1,\dots,A_m$ are two partitions of the same set, we say that the partition $\{P_1,\dots,P_k\}$ \emph{$\beta$-refines} $\{A_1,\dots,A_m\}$ if, for all but at most $\beta k$ indices $s \in [k]$, there exists some $i \in [m]$ such that $P_s \subset_\beta A_i$. For $\beta=0$, this is precisely the standard notion of refinement of partitions.
    We also remark that as long as $\beta<\frac 12$, there can be no ambiguity: if $P_s \subset_\beta A_i$, then $P_s {\nsubset_\beta} A_j$ for any $j \neq i$. We will ensure that $\beta<\frac 12$ in all that follows.

    \begin{lemma}\label{lem:beta-refinement}
        Let $P_1,\dots,P_k,Q_1,\dots,Q_k,R_1,\dots,R_k$ be partitions of $A,B,C$, respectively, and suppose that all parts have the same size up to a factor of two. 
        Suppose that all but $\varepsilon k^3$ of the triples $(P_a,Q_b, R_c)$ are weakly $\varepsilon$-regular in $H$. Suppose that $\varepsilon^{1/4} \leq \beta \leq \frac{1}{72}$. If $(P_1,\dots,P_k)$ and $(Q_1,\dots,Q_k)$ $\beta$-refine $\mathcal A_{r-1}$ and $\mathcal B_{r-1}$, respectively, then they $(7\beta)$-refine $\mathcal A_{r},\mathcal B_{r}$, respectively.
    \end{lemma}
    We stress that in the statement of this lemma, we do not care about how the partition $(R_1,\dots,R_k)$ interacts with the given partition $(C_1,\dots,C_t)$ of $C$; all that matters is the fact that $(P_1,\dots,P_k)$ and $(Q_1,\dots,Q_k)$ roughly refine $\mathcal A_{r-1}$ and $\mathcal B_{r-1}$.
    \begin{proof}[Proof of \cref{lem:beta-refinement}]
        We prove the statement for $A$; the statement for $B$ is completely analogous. 
        As before, write $\mathcal{A}_{r-1} = \{A_1,\dots,A_m\}$ and $\mathcal{A}_r = \{A_{i,j} : i \in [m], j \in [M]\}$, where $m=m_{r-1}$ and $M \coloneqq  \frac{m_r}{m_{r-1}}$.
        Fix some $P_s$, and suppose that $P_s \subset_\beta A_i$ for some $i \in [m]$. Suppose that there is no $j \in [M]$ such that $P_s \subset_{7\beta} A_{i,j}$. For each $j \in [M]$, let $\mu_j \coloneqq \ab{P_s \cap A_{i,j}}/\ab{P_s}$. Note that $$\mu \coloneqq  \sum_{j=1}^M \mu_j = \sum_{j=1}^M \frac{\ab{P_s\cap A_{i,j}}}{\ab{P_s}} = \frac{\ab{P_s \cap A_i}}{\ab{P_s}}\geq 1-\beta,$$ since $P_s \subset_\beta A_i$. Additionally, $\mu_j \leq 1-7\beta$ for every $j \in [M]$ by the assumption that there is no $j$ with $P_s \subset_{7\beta} A_{i,j}$. 
        Now set $\lambda_j \coloneqq  \mu_j/\mu$, so that $\sum_{j=1}^M \lambda_j = 1$ and $\lambda_j \leq \frac{\mu_j}{1-\beta} \leq \frac{1-7\beta}{1-\beta} \leq 1-6\beta$ for every $j \in [M]$.
        By Item 2 of \cref{lem:orthogonal partitions} with $\zeta \coloneqq 6\beta$, $\eta \coloneqq  5\beta$ and $2\varepsilon$ in place of $\varepsilon$ (the condition $(1-\eta)(1-8\varepsilon) \geq 1-\zeta+\zeta^2$ in Lemma \ref{lem:orthogonal partitions} holds since $\beta \leq \frac{1}{72}$ and $\varepsilon \leq \frac{\beta}{16}$), we get that there is $I \subseteq [m]$ with 
        $|I| \geq 5\beta m$, such that for every $h \in I$ we have 
        \begin{equation}\label{eq:min lambdas}
        \min \left(\sum_{j \in X_h} \lambda_j, \sum_{j \in Y_h} \lambda_j\right) > 2\varepsilon.
        \end{equation}
        Recall that for at least $(1-\beta)k$ choices of $u \in [k]$, there is some $h \in [m]$ such that $Q_u \subset_\beta B_h$. 
        We claim that for at least $\beta k$ of the indices $u \in [k]$, there is $h \in I$ with $Q_u \subset_\beta B_h$. Indeed, suppose otherwise. Then for all but $2\beta k$ of the indices $u \in [k]$, it holds that $|Q_u \cap B_h| \geq (1-\beta)|Q_u|$ for some $h \in [m] \setminus I$. 
        Also, we have $|Q_u| \leq \frac{2n}{k}$ for every $u \in [k]$ by the assumption that the sizes of the parts $Q_1,\dots,Q_k$ differ by at most a factor of 2. It follows that 
        $$n = \sum_{u=1}^k |Q_u| \leq 2\beta k \cdot \frac{2n}{k} + \frac{1}{1-\beta}\sum_{h \in [m] \setminus I}|B_h| = 4\beta n + \frac{1}{1-\beta} (m - |I|) \cdot \frac{n}{m} \leq 4\beta n + \frac{1-5\beta}{1-\beta}n < n,$$ a contradiction. This proves our claim. 
        
        Now fix any $u \in [k]$ such that $Q_u \subset_\beta B_h$ for some $h \in I$ (there are at least $\beta k$ choices for $u$ by the above). Let $V_X \coloneqq P_s \cap \bigcup_{j \in X_h} A_{i,j}, V_Y \coloneqq P_s \cap \bigcup_{j \in Y_h} A_{i,j}$, and similarly $W_X \coloneqq Q_u \cap \bigcup_{j \in X_i} B_{h,j}, W_Y \coloneqq Q_u \cap \bigcup_{j \in Y_i} B_{h,j}$. By construction and \eqref{eq:min lambdas}, we have that $$
        \min(\ab{V_X}, \ab{V_Y}) = 
        \min \left(\sum_{j \in X_h} \mu_j, \sum_{j \in Y_h} \mu_j\right) \cdot \ab{P_s} \geq \min \left(\sum_{j \in X_h} \lambda_j, \sum_{j \in Y_h} \lambda_j\right) \cdot (1-\beta) \cdot \ab{P_s} \geq 
        \varepsilon \ab{P_s}.
        $$
        Additionally, $W_X$ and $W_Y$ are disjoint subsets of $Q_u$ with $\ab{W_X \cup W_Y} \geq (1-\beta)\ab{Q_u}$, hence $\max(\ab{W_X},\ab{W_Y}) \geq (1-\beta)\ab{Q_u}/2 \geq \varepsilon \ab{Q_u}$. Suppose that $\ab{W_X} \geq \varepsilon \ab{Q_u}$; the other case is handled symmetrically.

        In the graph $G_r$, we have that $(V_X, W_X)$ is a complete pair, whereas $(V_Y,W_X)$ is empty. This implies that $d_H(V_X,W_X,C_r)=2^{-r}$ and $d_H(V_Y,W_X,C_r)=0$.
        For $\ell \in [k]$, let us say that $\ell$ is \emph{$r$-small} if $\ab{R_\ell \cap C_r} < \varepsilon \ab{R_\ell}$, where we recall that $R_1,\dots,R_k$ is the given partition of $C$. We say that $\ell$ is \emph{$r$-large} if it is not $r$-small. We have
        \[
        \ab{C_r} =  \sum_{\ell \text{ $r$-large}} \ab{R_\ell \cap C_r} + \sum_{\ell \text{ $r$-small}} \ab{R_\ell \cap C_r} < \sum_{\ell \text{ $r$-large}} \ab{R_\ell \cap C_r} + \sum_{\ell=1}^t \varepsilon \ab{R_\ell} = \sum_{\ell \text{ $r$-large}} \ab{R_\ell \cap C_r} + \varepsilon \ab C.
        \]
        Using that $\ab{C_r}=\ab C/t$, we get
        \[
        \sum_{\ell \text{ $r$-large}} \ab{R_\ell \cap C_r} > \frac{\ab{C}}{t} - \varepsilon \ab C \geq \frac{\ab C}{2t},
        \]
        where we used that $\varepsilon < 1/(2t)$, which holds because $t \leq \log \frac 1 \varepsilon$ and as $\varepsilon$ is assumed small enough.
        
        On the other hand, since all parts $R_\ell$ have the same size up to a factor of two, we have that
        \[
        \sum_{\ell \text{ $r$-large}} \ab{R_\ell \cap C_r} \leq \sum_{\ell \text{ $r$-large}} \ab{R_\ell} \leq \frac{2\ab C}{k} \cdot \ab{\{\ell \in [k]:\ell\text{ is $r$-large}\}}.
        \]
        Combining these two inequalities, and using that $\sqrt \varepsilon < 1/(4t)$ (again since $\varepsilon$ is sufficiently small), we find that
        \[
        \ab{\{\ell \in [k]:\ell\text{ is $r$-large}\}} \geq \frac{k}{4t} > \sqrt \varepsilon k.
        \]
        The final observation is that if $\ell$ is $r$-large, then the triple $(P_s, Q_u,R_\ell)$ is not weakly $\varepsilon$-regular. Indeed, we have that
        \[
        d_H(V_X,W_X,R_\ell \cap C_r) - d_H(V_Y,W_X,R_\ell \cap C_r) = 2^{-r} \geq 2^{-t} > \varepsilon,
        \]
        while the sets $V_X,V_Y,W_X,R_\ell \cap C_r$ are subsets of $P_s,P_s,Q_u,R_\ell$, respectively, each with at least an $\varepsilon$-fraction of the vertices of the corresponding set.

        Recall that this argument worked for each of at least $\beta k$ choices for $u$, and at least $\sqrt \varepsilon k$ choices for $\ell$. Therefore, under the assumption that $P_s\subset_\beta A_i$ but there is no $j$ such that $P_s \subset_{7\beta} A_{i,j}$, we find that there are at least $\beta \sqrt \varepsilon \cdot k^2$ choices of $(u,\ell) \in [k]^2$ for which the triple $(P_s,Q_u,R_\ell)$ is not weakly $\varepsilon$-regular. This can happen for at most $\beta k$ values of $s$, because the given partitions $(P_i,Q_i,R_i)_{i=1,\dots,k}$ are weakly $\varepsilon$-regular, and as $\beta > \varepsilon^{1/4}$. Also, at most $\beta k$ choices of $s$ do not satisfy $P_s \subset_\beta A_i$ for any $i$. Hence, we conclude that $\{P_1,\dots,P_k\}$ does indeed $(7\beta)$-refine $\mathcal A_r$.
    \end{proof}
    \noindent It is now a simple matter to complete the proof of \cref{thm:lower bound}.
    \begin{proof}[Proof of Theorem \ref{thm:lower bound}]
        Let $H$ be the weighted 3-partite 3-graph defined in Section \ref{sec:construction}. By Lemma \ref{lem:regular partitions in links}, every link $L_H(v)$ in $H$ has a $\delta$-regular partition of size $O(\delta^{-8})$. Next, we show that every weakly $\varepsilon$-regular equipartition of $H$ has size at least $\twr\left( \Omega(\log \frac 1 \varepsilon) \right)$.
        So fix such an equipartition, given by partitions $P_1,\dots,P_k,Q_1,\dots,Q_k,R_1,\dots,R_k$ of $A,B,C$, respectively. Setting $\beta_r \coloneqq 7^r \varepsilon^{1/4}$, we claim that for $0 \leq r \leq t$, the partitions $(P_1,\dots,P_k)$ and $(Q_1,\dots,Q_k)$ $\beta_r$-refine $\mathcal{A}_r$ and $\mathcal{B}_r$, respectively. Indeed, this follows by induction on $r$: the base case $r=0$ is trivial, and the induction step is given by Lemma \ref{lem:beta-refinement}. By the choice of $t = \frac{1}{4}\log_7(\frac{1}{\varepsilon}) - 3$, we get that $\beta_r < \frac{1}{72}$ for every $r$, as required by Lemma \ref{lem:beta-refinement}. It is easy to see that if an equipartition $\mathcal{P}$ $\beta$-refines an equipartition $\mathcal{A}$ for $\beta < \frac{1}{2}$, then $|\mathcal{P}| \geq |\mathcal{A}|/2$. Hence, $k \geq m_t/2$. Now, observe that $m_t \geq \twr\left( \Omega(\log \frac 1 \varepsilon) \right)$. Indeed, by the definition of the sequence $m_0,\dots,m_t$, we have $m_r = m_{r-1} \cdot \phi(m_{r-1})$ for all but at most one choice of $r$. Also, for all but $O(1)$ choices of $r$, we have $\phi(m_{r-1}) = \lfloor e^{m_{r-1}/16} \rfloor$, meaning that $m_r$ is exponential in $m_{r-1}$. It is easy to see that this implies that $m_t \geq \twr\left( \Omega(t) \right)$, as required.

        So far we proved Theorem \ref{thm:lower bound} for weighted 3-graphs. To obtain a (non-weighted) 3-graph, we let $H^*$ be the 3-partite 3-graph on $A,B,C$ obtained by sampling each edge $e \in A \times B \times C$ with probability $H(e)$, where $H(e)$ is the weight of $e$ in $H$. Using a Chernoff-type bound (e.g.\ Lemma \ref{lem:Hoeffding}), it is easy to see that with high probability the following holds:
         \begin{itemize}
            \item For all $X \subseteq A, Y \subseteq B, Z \subseteq C$ with $|X|,|Y|,|Z| = \Omega(n)$, it holds that $|d_{H^*}(X,Y,Z)- d_{H}(X,Y,Z)| = o(1)$.
            \item For all $v \in A$ and $Y \subseteq B, Z \subseteq C$ with $|Y|,|Z| = \Omega(n)$, 
            $|d_{L_{H^*}(v)}(Y,Z) - d_{L_{H}(v)}(Y,Z)| = o(1)$, and the analogous statement holds for $v \in B$ and $v \in C$.
        \end{itemize}
        Consequently, $H^*$ satisfies the assertion of the theorem with $\varepsilon,\delta$ replaced with $\varepsilon/2,2\delta$, respectively.
    \end{proof}

    \section{Concluding remarks}\label{sec:conclusion}
    In this section we record three corollaries of our results, none of which we believe to be quantitatively optimal. It would be very interesting to improve the quantitative aspects of any of them. We restrict our attention to $3$-graphs, but the corollaries and questions have natural analogues in all higher uniformities.

    \begin{proposition}\label{prop:regular implies homogeneous}
        For every $D \geq 1$ and $\varepsilon > 0$ there exists $\delta > 0$ such that the following holds. Let $H$ be a $3$-partite $3$-graph, and suppose that $H$ has slicewise VC-dimension at most $D$. If $H$ is weakly $\delta$-regular, then $d(H) \in [0,\varepsilon] \cup [1-\varepsilon,1]$.
    \end{proposition}
    In other words, if $H$ is both weakly $\delta$-regular and of bounded slicewise VC-dimension, then $H$ must be almost empty or almost complete. \cref{prop:regular implies homogeneous} follows immediately from \cref{thm:VC1}, as the latter implies that $H$ has a $\varepsilon$-homogeneous partition into at most $2^{\poly(1/\varepsilon)}$ parts. If $H$ is weakly $\delta$-regular for $\delta < 2^{-\poly(1/\varepsilon)}$, then in fact the density of each of these $\varepsilon$-homogeneous triples must be very close to the global density of $H$, implying that $H$ itself is very sparse or very dense.

    This argument shows that we may take $\delta$ to depend single-exponentially on $\varepsilon$ in \cref{prop:regular implies homogeneous}. However, we conjecture that this bound can be improved to polynomial.
    \begin{conjecture}\label{conj:regular homog poly}
    We may take $\delta = \varepsilon^{O(1)}$ in \cref{prop:regular implies homogeneous}.
    \end{conjecture}
    A useful perspective on \cref{conj:regular homog poly} has to do with embedding induced sub(hyper)graphs, which we now discuss.
    Let $B$ be a bipartite graph. We say that $B$ is a \emph{bipartitely induced subgraph} of a graph $G$ if $G$ contains an induced subgraph isomorphic to some supergraph of $B$, obtained by only adding new edges in the two parts of $B$, without adding or removing any edges across the bipartition. It is not hard to see that $G$ has bounded VC-dimension if and only if there is some fixed $B$ which is not a bipartitely induced subgraph of $G$. 

    This immediately implies that a $3$-graph $H$ has bounded slicewise VC-dimension if and only if it avoids $T$ as a tripartitely induced subhypergraph, where $T$ is some tripartite $3$-graph one of whose three parts is a singleton. In this language, we see that \cref{prop:regular implies homogeneous} can be viewed as an induced embedding lemma: in contrapositive, it says that if $H$ is weakly $\delta$-regular and has edge density bounded away from $0$ and $1$, then $H$ contains a tripartitely induced copy of any fixed tripartite $T$ one of whose parts is a singleton. One motivation for \cref{conj:regular homog poly} is that most embedding lemmas that one encounters can be proved with polynomial dependencies. Moreover, one could try to prove such an induced embedding lemma by first proving a counting lemma with two-sided error bounds, and then applying inclusion-exclusion. However, an unusual feature of this problem is that this corresponding counting lemma is simply false: it is well-known (see e.g.\ \cite{MR2595699,MR2864650}) that weak $\delta$-regularity is too weak to support a counting lemma (with two-sided error bounds) for any $T$ that is not linear, and hence proving \cref{conj:regular homog poly} would require a completely different technique.
    Additionally, in most settings when one can prove a counting lemma, it holds in any ``orientation'': namely, we should be able to embed the three parts of $T$ into the three parts of $H$ according to any bijection $[3]\to [3]$. However, a simple example shows that this is not true for \cref{prop:regular implies homogeneous}. Namely, consider the $3$-partite $3$-graph $H$ on parts $A,B,C$ obtained as follows: each vertex $c \in C$ picks uniformly random subsets $A_c \subseteq A, B_c\subseteq B$, and then $E(H)=\{(a,b,c):c \in C,a \in A_c, b \in B_c\}$. With high probability $H$ is weakly $o(1)$-regular and has density $\frac 14 +o(1)$. But if we let $T$ be the $3$-graph on $\{x,y_1,y_2,z_1,z_2\}$ with edge set $E(T) = \{(x,y_1,z_1),(x,y_2,z_2)\}$, then it is easy to see that $H$ contains no tripartitely induced copy of $T$ for which $x \in C$. This is not a contradiction to \cref{prop:regular implies homogeneous}, as there \emph{are} induced copies of $T$ with $x \in A\cup B$, but this example demonstrates another subtle feature of this unusual induced counting lemma, and shows why proving \cref{conj:regular homog poly} may be more challenging than proving other similar-looking induced embedding lemmas.
    
    Another natural statement that follows directly from \cref{thm:VC1}, and which also suggests an avenue for attacking \cref{conj:regular homog poly}, is the following partite hypergraph analogue of R\"odl's theorem \cite{MR837962} on induced $F$-free graphs. To make the analogy clearer, we continue using the language of forbidden tripartitely induced subgraphs, rather than discussing bounded slicewise VC-dimension.

    \begin{proposition}\label{prop:EH}
     Let $T$ be a tripartite $3$-graph one of whose parts is a singleton. For every $\varepsilon>0$, there exists $\delta>0$ such that the following holds.
        If $H$ is a $3$-partite $3$-graph on parts $V_1 \cup V_2 \cup V_3$, each of size $n$, and if $H$ contains no tripartitely induced copy of $H$, then there exist $V_i' \subseteq V_i$ with $\ab{V_i'} \geq \delta n$ such that $d(V_1',V_2',V_3') \in [0,\varepsilon] \cup [1-\varepsilon,1]$.
    \end{proposition}
    Indeed, \cref{prop:EH} follows immediately from \cref{thm:VC1}, as we may take $(V_1',V_2',V_3')$ to be any $\varepsilon$-homogeneous triple from the $\varepsilon$-homogeneous partition of $H$. In particular, \cref{thm:VC1} implies that we may take $\delta \geq 2^{-\poly(1/\varepsilon)}$ in \cref{prop:EH}. We conjecture that this dependence can also be improved to polynomial.
    \begin{conjecture}\label{conj:EH}
        We may take $\delta = \varepsilon^{O(1)}$ in \cref{prop:EH}.
    \end{conjecture}
    Note that \cref{conj:EH} would immediately imply \cref{conj:regular homog poly}. Moreover, it is an appealing statement for other reasons; for example, \cref{conj:EH} would imply the Erd\H os--Hajnal conjecture for $3$-partite $3$-graphs with no tripartitely induced copy of $T$. That is, assuming \cref{conj:EH}, a standard argument (see e.g.\ \cite[Conjecture 7.1]{MR2455625}) shows that if $H$ has no tripartitely induced copy of $T$, then there exist $V_i'' \subseteq V_i$ such that $\ab{V_i''} = n^{\Omega(1)}$, and such that the triple $(V_1'',V_2'',V_3'')$ is either complete or empty. We remark that one must work in this tripartite setting in order to obtain such an Erd\H os--Hajnal-type result: \cite[Theorem 1.2]{MR3217709} implies that there exist $H$ with no tripartitely induced copy of $T$ and with no clique or independent set of size greater than $\Theta(\log n)$.

    Nonetheless, we are still able to prove an analogue of R\"odl's theorem \cite{MR837962} in the non-partite setting, as follows.
    \begin{proposition}\label{prop:rodl}
        Let $T$ be a tripartite $3$-graph one of whose parts is a singleton. For every $\varepsilon>0$, there exists $\delta>0$ such that the following holds. If $H$ is a $3$-graph with no tripartitely induced copy of $T$, then there exists $U \subseteq V(H)$ with $\ab U \geq \delta \ab{V(H)}$ such that $d(U) \in [0,\varepsilon]\cup [1-\varepsilon,1]$.
    \end{proposition}
    The deduction of \cref{prop:rodl} is relatively standard, so we only sketch it. First, by \cref{thm:VC1} and \cref{rem:partite to general}, we can find an $\varepsilon$-homogeneous partition of $H$,
    as the assumption that there is no tripartitely induced copy of $T$ implies that all links have bounded VC-dimension. 
    We now use Tur\'an's theorem for hypergraphs to pass to a large number of parts containing no non-homogeneous triples, and then apply Ramsey's theorem for hypergraphs to pass to yet another subset in which either all triples have density at most $\varepsilon$, or all triples have density at least $1-\varepsilon$. The union of these parts is then the desired set $U$. 

    The proof sketched above shows that $\delta$ can be taken to depend triple-exponentially on $\varepsilon$ (one exponential from the application of \cref{thm:upper bound}, then two more from the application of the $3$-uniform Ramsey theorem). We are uncertain about the true dependence.
    \begin{problem}
    What is the optimal dependence of $\delta$ on $\varepsilon$ in \cref{prop:rodl}?
    \end{problem}

    \paragraph{Acknowledgments:} Part of this work was conducted while the authors were visiting the SwissMAP research center at Les Diablerets. We thank them for their welcome and for providing an excellent working environment. We also thank Artem Chernikov and the anonymous referees for helpful comments on an earlier version of this paper.
\bibliographystyle{yuval}
\bibliography{refs.bib}

\appendix

\section{Proof of Lemma \ref{lem:orthogonal partitions}}\label{sec:appendix}
We will need Hoeffding's inequality (see \cite[Theorem A.1.4]{AlonSpencer}):
    \begin{lemma}[Hoeffding's inequality]\label{lem:Hoeffding}
    Let $Z \sim \Bin(n,p)$. Then 
    $$
    \mathbb{P}[Z \geq np + t],\mathbb{P}[Z \leq np - t] \leq e^{-2t^2/n}.
    $$
    \end{lemma}

    \begin{proof}[Proof of Lemma \ref{lem:orthogonal partitions}]
    We choose each partition $(X_i,Y_i)$, $1 \leq i \leq m$, randomly, by including each element of $[M]$ into one of the sets $X_i,Y_i$ with probability $\frac{1}{2}$, independently of all other choices. Thus both $|X_i|$ and $|Y_i|$ are distributed as $\Bin(M,\frac 12)$, and in particular $\mathbb{E}[|X_i|] = \mathbb{E}[|Y_i|] = \frac{M}{2}$. By Lemma \ref{lem:Hoeffding}, the probability that $| |X_i| - \frac{M}{2}| \geq M^{2/3}$ is at most $2e^{-2M^{1/3}}$. Similarly, fixing $1 \leq i < i' \leq m$, each of $|X_i \cap X_{i'}|, |X_i \cap Y_{i'}|, ,|Y_i \cap X_{i'}|, |Y_i \cap Y_{i'}|$ has expected value $\frac{M}{4}$, and by Lemma \ref{lem:Hoeffding}, the probability that it deviates from its expectation by more than $M^{2/3}$ is at most $2e^{-2M^{1/3}}$. 
    There are $4\binom{m}{2}$ such events (4 per each pair $1 \leq i < i' \leq m$).
    By taking the union bound over all $1 \leq i \leq m$ (for the first statement) and $1 \leq i < i' \leq m$ (for the second statement), we get that the probability that Item 1 fails is at most $(m + 4\binom{m}{2}) \cdot 2e^{-2M^{1/3}} \leq 
    2m^2 \cdot 2e^{-2M^{1/3}} \leq \frac 12$, where the inequality holds if $M \geq \log^3(8m^2)$. Thus, Item 1 holds with probability at least $\frac 12$.

    Now consider Item 2. 
    For distinct $j,j' \in [M]$, let $z_{j,j'}$ be the number of indices $i$ for which $j,j'$ are on the same side of the partition $(X_i,Y_i)$.
    Let $\mathcal{A}$ be the event that $z_{j,j'} \leq \frac{3m}{4}$ for all $j,j'$. We claim that $\mathbb{P}[\mathcal{A}] > \frac 12$. Indeed, fixing $j,j'$, we have $z_{j,j'} \sim \Bin(m,\frac{1}{2})$. By Lemma \ref{lem:Hoeffding}, $\mathbb{P}[z_{j,j'} \geq 3m/4] \leq e^{-2 \cdot (1/4)^2 m} = e^{-m/8}$. By the union bound over $j,j'$, the probability that $\mathcal{A}$ fails is at most $\binom{M}{2} \cdot e^{-m/8} < \frac 12$, using that $M \leq e^{m/16}$. 
    
    We conclude that with positive probability, Item 1 and $\mathcal{A}$ hold. Now we show that $\mathcal{A}$ implies Item 2. Put $g_i(\lambda) \coloneqq  \left| \sum_{j \in X_i} \lambda_j - \sum_{j \in Y_i} \lambda_j \right|$. We have
    \begin{align*}
    \sum_{i=1}^m g_i(\lambda)^2 &= \sum_{i=1}^m \left[ \sum_{j=1}^M \lambda_j^2 \; \; + 
    \sum_{j,j' \in \binom{X_i}{2} \cup \binom{Y_i}{2}} 2\lambda_j \lambda_{j'} \; - 
    \sum_{j\in X_i, j' \in Y_i} 2\lambda_j\lambda_{j'} \right] \\
    &=
    m \sum_{j=1}^M \lambda_j^2 + \sum_{j,j' \in \binom{M}{2}}\left(4z_{j,j'} - 2m\right)\lambda_j \lambda_{j'} \leq 
    m \left( \sum_{j=1}^M \lambda_j^2 + \sum_{j,j' \in \binom{M}{2}} \lambda_j \lambda_{j'} \right) 
    \\ &= 
    \frac{m}{2}\left( \sum_{j=1}^M \lambda_j^2 + \left( \sum_{j=1}^M \lambda_j \right)^2 \right) = 
    \frac{m}{2}\left( \sum_{j=1}^M \lambda_j^2 + 1 \right),
    \end{align*}
    where the inequality uses that $\mathcal{A}$ happened. Convexity implies that under the constraints $0 \leq \lambda_j \leq 1-\zeta$ for every $j$ and $\sum_{j=1}^M \lambda_j = 1$, the function $\sum_{j=1}^M \lambda_j^2$ is maximized when some $\lambda_j$ equals $1-\zeta$ and another $\lambda_j$ equals $\zeta$. Hence, $\sum_{j=1}^M \lambda_j^2 \leq (1-\zeta)^2 + \zeta^2 = 1-2\zeta + 2\zeta^2$. Plugging this into the above, we get that 
    $
    \sum_{i=1}^m g_i(\lambda)^2 \leq m(1-\zeta + \zeta^2) \leq m(1-\eta)(1-4\varepsilon).
    $ 
    Hence, the number of $1 \leq i \leq m$ with $g_i(\lambda)^2 \geq 1-4\varepsilon$ is at most $(1-\eta)m$. If $g_i(\lambda)^2 < 1-4\varepsilon$ then $g_i(\lambda) < 1-2\varepsilon$, which can only happen if $\min\left( \sum_{j \in X_i}\lambda_j, \sum_{j \in Y_i} \lambda_j \right) > \varepsilon$. This proves that Item 2 in the lemma holds.   
    \end{proof}

\end{document}